\documentclass{amsart}
\usepackage[margin=2.4cm]{geometry}

\usepackage{graphics, amsmath,amssymb, enumerate, comment, url, mathtools}
\usepackage{graphicx}
\usepackage{epsfig}

\newif\ifcolorcomments
\newcommand{\allowcomments}[4]{
\newcommand{#1}[1]{\ifdraft{\ifcolorcomments{\textcolor{#4}{##1 --#3}}\else{\textsl{ ##1 \ --#3}}\fi}\else{}\fi}
}

\colorcommentstrue
\usepackage{amssymb}
\usepackage{amsmath,amsthm}

\usepackage{colortbl}

\newtheorem{theorem}{Theorem}[section]
\newtheorem{lemma}[theorem]{Lemma}
\newtheorem{proposition}[theorem]{Proposition}
\newtheorem{corollary}[theorem]{Corollary}
\theoremstyle{definition}

\newcommand{\CC}{\mathcal C}

\newcommand{\DD}{\mathcal D}

\newcommand{\HH}{\mathcal H}

\newcommand{\N}{\mathbb N}% by default, $\N$ is American naturals $\{1,2,\ldots\}$
\newcommand{\Z}{\mathbb Z}

\newcommand{\mbf}{\mathbf}
\newcommand{\ff}{\mbf f}

\renewcommand{\text}{\textup}

\newcommand{\NPC}[1]{\ignorespaces}

\DeclareMathOperator{\Var}{Var}

\newif\ifdraft\drafttrue

\def\N{\mathbb N}
\def\Z{\mathbb Z}

\IfFileExists{marvosym.sty}{
\RequirePackage{marvosym} 
}{

}

\IfFileExists{wasysym.sty}{
\RequirePackage{wasysym}\renewcommand{\emptyset}{{\diameter}}
}{

}

\newcommand*{\myDots}{\ifmmode\mathellipsis\else.\kern-0.07em.\kern-0.07em.\fi}

\allowcomments{\commumtaz}{MH}{Mumtaz}{green}
\allowcomments{\comnikita}{NS}{Nikita}{blue}
\allowcomments{\combixuan}{BL}{Bixuan}{red}

\newcommand {\ignore}[1] {}

\newcommand{\commenty}[1]{}

\begin{document}

\title{Metrical properties of finite product of partial quotients in arithmetic progressions}

\author[Mumtaz Hussain]{Mumtaz Hussain}
\address{Mumtaz Hussain,  Department of Mathematical and Physical Sciences,  La Trobe University, Bendigo 3552, Australia. }
\email{m.hussain@latrobe.edu.au}

%\author[Bixuan Li]{Bixuan Li}
%\address{Bixuan Li, School of Mathematics and Statistics, Huazhong University of Science and Technology, Wuhan 430074, P. R. China}
%\email{libixuan@hust.edu.cn}

\author{Nikita Shulga}
\address{Nikita Shulga,  Department of Mathematical and Physical Sciences,  La Trobe University, Bendigo 3552, Australia. }
\email{n.shulga@latrobe.edu.au}
\date{}

\maketitle

\numberwithin{equation}{section}

\begin{abstract}
We investigate the dynamics of continued fractions and explore the ergodic behaviour of the products of mixed partial quotients in continued fractions of real numbers. For any function $\Phi:\N\to [2,+\infty)$ and any integer $d\geq 1$, we determine the Lebesgue measure and Hausdorff dimension of the set of real numbers for which the product of partial quotients in arithmetic progressions satisfy $a_n(x)a_{2n}(x)\cdots a_{dn}(x)\geq \Phi(n)$ for infinitely many positive integers $n$.

Our findings shed light on the size of the set of exceptions to Bourgain's (1988) and Host and Kra's (2005) theorems concerning the convergence of multiple ergodic averages for Gauss dynamical systems. By exploring the Hausdorff dimension of these sets, we gain valuable insights into the behaviour of such exceptions. Overall, our research contributes to a deeper understanding of the dynamics of continued fractions and their connection to the convergence properties of ergodic averages in Gauss dynamical systems.
\end{abstract}

\section{Introduction}
Let $\left(X,\mathcal{B},\mu,T\right)$ be a measure preserving dynamical system, where $\mathcal{B}$ is a Borel $\sigma$-algebra over $X$, $\mu$ is a probability measure and $T:X\to X$ is an invertible transformation which preserves the measure. 
A fundamental problem in ergodic theory is to understand whether the ergodic average \begin{align*}
\frac{1}{N}\sum_{n=0}^{N-1}f\left(T^{n}\left(x\right)\right)
\end{align*}
converges everywhere or almost everywhere for any function $f$ from some class of functions.  This question is answered by the famous Birkhoff ergodic theorem, which states that the above average converges pointwise for $\mu$-almost every $x\in X$ for any $f\in L^1(X,\mu)$. 

A new impulse for studying ergodic averages was given by Furstenberg when he provided a new proof of Szemer\'edi Theorem by using ergodic theoretic methods affirming the existence of arithmetic sequences of arbitrary length amongst sets of integers with positive density, see \cite{Furstenberg77}.   The multiple recurrence theorem is stated as follows:
\begin{theorem}[Furstenberg, 1977]\label{Fur97}
Let $\left(X,\mathcal{B},\mu,T\right)$ be a measure preserving dynamical system such that $\mu(A)>0$ for any $A\subset X$. Then for any $d\in\mathbb{N}$, we have
\begin{equation}
\liminf_{N\rightarrow\infty}\frac{1}{N}\sum_{n=0}^{N-1}\mu\left(A\cap T^{-n}A\cap T^{-2n}A\cap\cdots\cap T^{-dn}A\right)>0.
\end{equation} 
\end{theorem}
Afterwards, many authors have studied extensions of classical ergodic theorems on `non-conventional' ergodic averages. Bourgain in \cite{MR0937581} provided the following extension to Birkhoff's theorem.
\begin{theorem}[Bourgain, 1988]\label{Bour88}
Let $\left(X,\mathcal{B},\mu,T\right)$ be a measure preserving dynamical system, and let $f\in L^p(X,\mu)$ for any $1<p<\infty$. Let $P\in \Z[n]$ be a polynomial with integer coefficients. Then the averages
\begin{align*}
\frac{1}{N}\sum_{n=0}^{N-1}f\left(T^{P(n)}\left(x\right)\right)
\end{align*}
converge pointwise for $\mu$-almost every $x\in X$.
\end{theorem}

Furstenberg's proof was also a starting point for studying the convergence of multiple ergodic averages. To be more precise, for any bounded measurable functions $f_1,\ldots, f_d$, we want to know whether the multiple ergodic averages
\begin{align*}
\frac{1}{N}\sum_{n=0}^{N-1}f_{1}\left(T^{n}\left(x\right)\right)\cdots f_{d}\left(T^{dn}\left(x\right)\right)
\end{align*}
converges everywhere (or almost everywhere)  in the $L^{2}$ norm.  Host and Kra \cite{HostKra} proved the following result.
\begin{theorem}[Host-Kra, 2005]\label{HKthm}
Let $\left(X,\mathcal{B},\mu,T\right)$ be a measure preserving dynamical system, $d$ an integer, and $f_1, \ldots, f_d$ be bounded measurable functions on $X$. Then the averages
\begin{equation}\label{Bour}
A_N(\ff)=\frac1N\sum\limits_{n=0}^{N-1} f_1(T^{n}(x)) f_2(T^{2n}(x))\cdots  f_d(T^{dn}(x))
\end{equation} 
converges almost surely in $L^2(\mu)$ as $n\to\infty$.
\end{theorem}
Note that, by taking $f_1=\cdots=f_d=\mathbf I_A$ and integrating \eqref{Bour} over $A$ we get that the $\liminf$ in Furstenberg's theorem is a limit. For a selection of results in this direction, readers may refer to publications such as \cite{DonSon, FSW, HostKra, KMT22}. Bergelson and Leibman \cite{BerLei96} proved that the exponents, $n, 2n, \ldots$ appearing in the Furstenberg theorem can be replaced by integer polynomials.

In this paper, we focus on the Gauss dynamical system of continued fractions and as a consequence of our main theorem, we find the Hausdorff dimension of the set associated with Bourgain and Host-Kra's theorems for this system. 

It is well-known that every irrational number $x\in[0,1)$ can be uniquely expanded into an infinite continued fraction
\begin{align*}
x=\cfrac{1}{a_{1}\left(x\right)+\cfrac{1}{a_{2}\left(x\right)+\cdots}}.
\end{align*}
More specifically, we write $x=\left[a_{1}\left(x\right),a_{2}\left(x\right),\ldots\right]$ for the continued fraction expansion of $x$, where $a_{1}\left(x\right)=\lfloor 1/x\rfloor$ and $a_{n}\left(x\right)=a_{1}\left(T^{n-1}\left(x\right)\right)$ for $n\ge 2$ are called the partial quotients of $x$. From here and on, by $T(x)$ we denote the Gauss map $T: [0,1]\to[0,1]$ defined by $T(0)=0$ and 
$$T(x) := \frac{1}{x} - \left\lfloor \frac{1}{x} \right\rfloor \quad\text{ for } x\neq0.$$
The finite truncation $p_{n}\left(x\right)/q_{n}\left(x\right)=\left[a_{1}\left(x\right),\ldots,a_{n}\left(x\right)\right]$ is called $n$th convergent of $x$. 

{Let $e^{f(x)}=f_{i}(x)=a_{1}(x)$ for any $i\ge 1$. Note that the function $a_1(x) \notin L^1\left([0,1],\mathcal{L}\right)$, so, for instance, we cannot directly apply Birkhoff's theorem. However, we can deal with this by considering `truncated' functions 
\begin{equation}
h_M(x)=\left\{
  \begin{array}{ll}
   f_i(x)  &\textrm{if } f_i(x) \leq M;\\
   0 & \textrm{if } f_i (x)>M,
  \end{array}
\right.   \text{    for   } M=1,2,\ldots, \\
\end{equation}
applying Birkhoff's theorem to them and then taking a limit at $M\to\infty$.
We refer the reader to \cite{MR610981} for more details. For this choice of $f(x)$ and the Gauss dynamical system, we study metrical properties of sets, in which products of partial quotients with indices in arithmetic progression, grow at a certain rate.

To formally define our setup, let $\Phi:\N\to[2,+\infty)$ be a positive function, for which we define
\begin{equation}\label{defPhi}\log B=\liminf_{n\rightarrow\infty}\frac{\log\Phi\left(n\right)}{n},\textrm{ }\log b=\liminf_{n\rightarrow\infty}\frac{\log\log\Phi\left(n\right)}{n}.
\end{equation}
Fix $d\geq 1$ and define the set
\begin{align*}
\Lambda_d\left(\Phi\right):=&\Big\{x\in\left[0,1\right):a_{n}(x)\cdots a_{nd}(x)\ge\Phi(n)\textrm{ for infinitely many }n\in\mathbb{N}\Big\}.
\end{align*}
Note that when $d=1$, the classical Borel-Bernstein theorem (1912) states that the Lebesgue measure of the set
\begin{equation*}
\Lambda_1(\psi):=\left\{x\in [0, 1): a_n(x)\geq \Phi(n) \ {\rm for \ infinitely \ many} \ n\in \N\right\}
\end{equation*}
is either zero or one depending upon the convergence or divergence of the series $\sum_{n=1}^\infty \Phi(n)^{-1}$ respectively. Taking $\Phi(n)=n\log n$, one can conclude that for almost all $x\in[0,1)$, with respect to Lebesgue measure, $a_n(x)\geq n\log n$ holds for infinitely many $n\in\mathbb{N}$. This implies that the law of large numbers does not hold, that is
\begin{equation}\label{LLN}
\lim_{N\to\infty}\frac{1}{N}\sum_{n<N} a_n(x)=\infty
\end{equation}
for almost all $x$.  Khintchine \cite{Khintchine1935} showed that a weak law of large number holds for a suitable normalising sequence, that is, $\sum_{i=1}^na_i(x)/n\log n$ converges to $1/\log 2$ with respect to the Lebesgue measure. Philipp \cite{Philipp88} proved that there is no reasonably regular function $\phi:\mathbb N\to\mathbb R_+$ such that
$\sum_{i=1}^na_i(x)/\phi(n)$ almost everywhere converges to a finite nonzero constant. However, Diamond and Vaaler \cite{DiamondVaaler} showed that the strong law of large numbers with the normalising sequence $n\log n$ holds if the largest partial quotient $a_k(x)$ is trimmed from the sum. 
The method of deleting some number of maximal terms from Birkhoff sum to obtain strong laws of large numbers is called \textit{trimming}. The first result using trimming in the dynamical systems context is the aforementioned result by Diamond and Vaaler, which was later generalised to other contexts. For example, Aaronson and Nakada \cite{AN2003} gave strong laws of large numbers under light trimming (that is, deleting a finite number of terms) for sufficiently fast $\psi$-mixing random variables.
For other applications, see \cite{KSch1, KSch2} for intermediate trimming (that is, deleting infinitely many terms, while the number of terms deleted should be $\overline o(n)$ compared with the length $n$ of Birkhoff sum) applied to subshifts of finite type or \cite{bonanno2023sure} for generalisation on other types of continued fractions.

Note that the divergence of (\ref{LLN}) implies the divergence of product of partial quotients, that is,  $$\lim_{N\to\infty}\frac{1}{N}\sum_{n<N} a_n(x)a_{2n}(x)\cdots a_{dn}(x)=\infty.$$
Hence the natural question is to find a normalising sequence $\varphi(n)$ such that $\frac{\sum_{i=1}^n a_i(x)a_{2i}(x)\cdots a_{di}(x)}{\varphi(n)}$ converges to a certain finite limit (in the spirit of the weak or strong law of large numbers stated above).  However, these can be proved by following the methods presented in \cite{HHY}. Nonetheless, as a consequence of the Lebesgue measure dichotomy statement below, one can ascertain that 
$$a_n(x)a_{2n}(x)\cdots a_{dn}(x)\geq n\log^2(n)$$ holds for infinitely many $n$. Thus, any $x\in \Lambda_d\left(\Phi\right)$ is an exception to Theorems \ref{Bour88} and \ref{HKthm}. Hence the natural question is to quantify this set. Our first result provides the Lebesgue measure of this set, depending on the function $\Phi(n)$.

\begin{theorem}\label{Lebthm} Let $\Phi:\N\to[2,+\infty)$ be a positive function. Then, the Lebesgue measure of the set  $\Lambda_d\left(\Phi\right)$ is given by

\begin{equation*}
\mathcal L\left(\Lambda_d\left(\Phi\right)\right)=\left\{ 
\begin{array}{ll}
0& {\rm if}\ \ \sum\limits_{n=1}^\infty \frac {\log^{d-1}\Phi(n)}{\Phi(n)}<\infty; \\ [3ex] 
1 & {\rm if}\ \  \sum\limits_{n=1}^\infty \frac {\log^{d-1} \Phi(n)}{\Phi(n)}=\infty.
\end{array}
\right.
\end{equation*}
\end{theorem}

Note that if $\Phi(n) \leq 1$  for infinitely many $n$, then the set $\Lambda_d(\Phi)$ has full Lebesgue measure, but the corresponding series may converge. To see this, one can take a function $\Phi(n) = 1$. As the partial quotients are integer numbers, the set $\Lambda_d(\Phi)$ is unchanged if we change $\Phi$ by $\lceil \Phi\rceil$. This is why we can assume that $\Phi(n) \geq 2$.

Note that for rapidly increasing $\Phi$ the Lebesgue measure of the set $\Lambda_d(\Phi)$ is zero. Hence, it is pertinent to discern between sets of Lebesgue measure zero where the Hausdorff dimension serves as a suitable measure. 
\begin{theorem}\label{FurstCor} Let $\Phi$ be as in Theorem \ref{Lebthm}, for which we define quantities \eqref{defPhi}. Then
\begin{equation*}
\dim_\HH \Lambda_d\left(\Phi\right)=\left\{ 
\begin{array}{ll}
1 & {\rm if} \ \ B=1;\\
\lambda_d(B) & {\rm if}\ \ 1<B<\infty; \\ [3ex] 
\frac{1}{1+b} & {\rm if} \ \  B=\infty,
\end{array}
\right.
\end{equation*}
where $\lambda_d(B)$ is a continuous decreasing function, such that $\lim_{B\rightarrow 1}\lambda_d(B)=1,\textrm{  }\lim_{B\rightarrow\infty}\lambda_d(B)=\frac{1}{2}.$ This means that the Hausdorff dimension continuously changes with respect to the values of $B$ and $b$.

The function $\lambda_d(B)$ is equal to the following expression in terms of the pressure function.
$$\lambda_d(B):=\inf\left\{s:\mathsf{P}\left(-s\log\left|T'(x)\right|-g_d(s)\log B\right)\le 0\right\},$$
where $g_n(s)$ is given by the iterative relation
\begin{align}\label{defGN}
g_1(s) = s, \,\,\,\,\, g_n(s)&:=\frac{sg_{n-1}(s)}{1-s+ng_{n-1}(s)}.
\end{align}
\end{theorem}

Here and throughout $T^{\prime }(x)$ denotes the derivative of the Gauss map $T(x)$, and $\mathsf{P}$ represents the pressure function defined in the Subsection \ref{Pressure Functions}. We prove the continuity and limit properties of $\lambda_d(B)$ in the same subsection. The dimensional number $\lambda_d(B)$ can also be expressed as a limit of solutions of some equations, as shown in Proposition \ref{pro2} below. For more details on the function $g_n(s),$ including an explicit (non-iterative) formula, see Section \ref{sec2}.

We also compare the Hausdorff dimension of the set $\Lambda_d(\Phi)$ with the dimension of the well-studied set of large products of consecutive partial quotients 
 \begin{align}\label{KKWW}
\mathcal{D}_{d}(\Phi)=\left\{x\in[0, 1): \prod_{i=1}^d a_{n+i}(x)\geq \Phi(n) \ \
\mathrm{for \ infinitely \ many \ }n\in\mathbb{N}\right\}.
\end{align}
We prove that in the most interesting case of $1<B<\infty$, the Hausdorff dimension of $\Lambda_d(\Phi)$ is larger than the Hausdorff dimension of $\mathcal{D}_{d}(\Phi)$ when $d\geq2$. For details, see Section \ref{comparison}.

\medskip

It is worth stating that the consideration of products of partial quotients arose from Kleinbock and Wadleigh paper \cite{KleinbockWadleigh} in which it was shown that considering the growth of the product of consecutive partial quotients gives information about the set of Dirichlet non-improvable numbers. Since then there have been several works in considering the sets of real numbers for which the product of (finite) consecutive partial quotients growing at a rate given by any positive function. We refer the reader to \cite{BBH1, BHS, HKWW, HLS,HussainShulga1, KleinbockWadleigh,  LWX} for a selection of results in this direction.

\medskip

The most challenging part of the proof is in establishing the lower bound of the Hausdorff dimension for Theorem \ref{FurstCor} in the case $1< B<\infty$. For all other cases, the proofs are relatively easier.  In proving the lower bound in this case, we will employ the  \emph{Mass distribution principle}. To do this,  we
\begin{itemize}
\item Define a good subset of $\Lambda_d\left(\Phi\right)$ in terms of union of fundamental intervals;
\item Calculate the lengths of fundamental cylinders and gaps between them;
\item Define (and calculate) the probability measure supported on a suitable subset (Cantor subset) of $\Lambda_d\left(\Phi\right)$ in terms of the diameter of the fundamental intervals and calculate the H\"older exponent;
\item Calculate the H\"older exponent for an arbitrary ball;
\item Apply the mass distribution principle to calculate the lower bound of the Hausdorff dimension.
\end{itemize}
In contrast to the consecutive partial quotient consideration, where authors have considered the growth of $d$ consecutive partial quotients, in our setting we deal with the product of $d$ partial quotients each arising from the arithmetic progression. This leads to significant complications in constructing a suitable Cantor subset, as we will need to introduce and optimise a  number of parameters.  This raises some difficulties in distributing the mass over fundamental cylinders, which we overcome in the course of the proof.  
\medskip

\noindent{\bf Acknowledgements} The research of Mumtaz Hussain and Nikita Shulga is supported by the Australian
Research Council Discovery Project (200100994). Most of this work was carried out when Mumtaz and Nikita visited Dzmitry Badziahin at the University of Sydney. We are grateful for the support and hospitality of the Sydney Mathematical Research Institute (SMRI). We thank Professor Baowei Wang and Bixuan Li for useful discussions. Finally, we thank an anonymous referee for the careful reading of the manuscript, and several useful comments that have increased the readability and clarity of proofs, and simplified the exposition of the paper.

\section{Preliminaries}\label{sec2}
In this section, we shall give some basic properties about continued fractions which we will use later.

Recall that the $n$th convergents of $x$ is defined as
\begin{align*}
\frac{p_{n}\left(x\right)}{q_{n}\left(x\right)}=\left[a_{1}\left(x\right),\ldots,a_{n}\left(x\right)\right].
\end{align*}
For simplicity, we will denote $p_n=p_{n}\left(a_{1},\ldots,a_{n}\right)$, $q_n=q_{n}\left(a_{1},\ldots,a_{n}\right)$ when there is no ambiguity. The rule for the formation of the convergents is as follows, for any $k\ge 2$ we have
\begin{equation}\label{rule}
\begin{split}
p_{k}&=a_{k}p_{k-1}+p_{k-2},\\
q_{k}&=a_{k}q_{k-1}+q_{k-2},
\end{split}
\end{equation}
where $p_{0}=0,p_{1}=1$ and $q_{0}=1,q_{1}=a_1$.

For $n\geq1$ and any $\left(a_{1},\ldots,a_{n}\right)\in\N^n$, where $a_{i}\in\mathbb{N}, 1\le i\le n$, we define $\textit{a cylinder of order }n$, denoted by $I_{n}\left(a_{1},\ldots,a_{n}\right)$ as
\begin{align*}
I_{n}\left(a_{1},\ldots,a_{n}\right):=\left\{x\in\left[0,1\right):a_{1}\left(x\right)=a_{1},\ldots,a_{n}\left(x\right)=a_{n}\right\}.
\end{align*}
We have the following standard facts about cylinders.
\begin{proposition}[{\protect\cite{Khinchin_book}}]\label{range}
For any $n\ge 1$ and $\left(a_{1},\ldots,a_{n}\right)\in\mathbb{N}^{n}$, $p_{k},q_{k}$ are defined recursively by $\left(\ref{rule}\right)$ where $0\le k\le n$, we have
\begin{equation}I_{n}\left(a_{1},\ldots,a_{n}\right)=\left\{
  \begin{array}{ll}
   \left[\dfrac{p_{n}}{q_{n}},\dfrac{p_{n}+p_{n-1}}{q_{n}+q_{n-1}}\right) &{\rm if  }
   \,n\textrm{ is even},\\
   \left(\dfrac{p_{n}+p_{n-1}}{q_{n}+q_{n-1}},\dfrac{p_{n}}{q_{n}}\right] &{\rm if    }\,n\textrm{ is odd}.\\
  \end{array}
\right.
\end{equation} 
Therefore, the length of a cylinder is given by
\begin{equation}\label{length}
\left|I_{n}\left(a_{1},\ldots,a_{n}\right)\right|=\dfrac{1}{q_{n}\left(q_{n}+q_{n-1}\right)}.
\end{equation}
\end{proposition}
For a better estimation, we need the following results. For any $n\ge 1$ and $\left(a_{1},\ldots,a_{n}\right)\in\mathbb{N}^{n}$, we have
\begin{lemma}[{\protect\cite{Khinchin_book}}] For any $1\le k\le n$, we have
\begin{equation}\label{c1}
\dfrac{a_{k}+1}{2}\le\dfrac{q_{n}\left(a_{1},\ldots,a_{n}\right)}{q_{n-1}\left(a_{1},\ldots,a_{k-1},a_{k+1},\ldots,a_{n}\right)}\le a_{k}+1.
\end{equation}
\end{lemma}
\begin{lemma}[{\protect\cite{Khinchin_book}}] For any $k\ge 1$, we have
\begin{equation}\label{c2}
\begin{split}
q_{n+k}\left(a_{1},\ldots,a_{n},a_{n+1},\ldots,a_{n+k}\right)&\ge q_{n}\left(a_{1},\ldots,a_{n}\right)\cdot q_{k}\left(a_{n+1},\ldots,a_{n+k}\right),\\
q_{n+k}\left(a_{1},\ldots,a_{n},a_{n+1},\ldots,a_{n+k}\right)&\le 2q_{n}\left(a_{1},\ldots,a_{n}\right)\cdot q_{k}\left(a_{n+1},\ldots,a_{n+k}\right).
\end{split}
\end{equation}
\end{lemma}
For any $n\ge 1$, since $p_{n-1}q_{n}-p_{n}q_{n-1}=\left(-1\right)^{n}$, combined with a simple calculation, we get the bound
\begin{equation}\label{qE}
q_{n}\ge 2^{\left(n-1\right)/2}.
\end{equation}

We will also use a trivial bound 
\begin{equation}\label{trivialbound}
a_1\cdots a_n \leq q_{n} \leq (a_1+1)\cdots (a_n+1).
\end{equation}

Next, we introduce the mass distribution principle, which is useful in calculating the lower bound of the Hausdorff dimension.
\begin{proposition}[{\protect\cite{Falconer_book2020}}]\label{MD}Let $E\subseteq\left[0,1\right)$ be a Borel set and $\mu$ be a measure with $\mu\left(E\right)>0$, suppose that for some $s>0$, there is a constant $c>0$ such that for any $x\in\left[0,1\right)$
\begin{equation}\label{MA}
\mu\left(B\left(x,r\right)\right)\le cr^{s},
\end{equation}
where $B\left(x,r\right)$ denotes an open ball centred at $x$ and radius $r$, then $\dim_{H}E\ge s$.
\end{proposition}

Let $\mu$ be the Gauss measure on [0,1] defined by
\[\mu(A)=\frac{1}{\log2 }\int_A\frac{1}{1+x}dx\]
for any Lebesgue measurable set $A$. The Gauss measure $\mu$ is $T$-invariant and equivalent to Lebesgue measure $\mathcal{L}$.
The following exponentially mixing property of Gauss measure is well known (see \cite{Billingsley} or \cite{Philipp88}).
\begin{lemma}\label{mixing-estimate}
  There exists a constant $0<\rho<1$ such that
  $$\mu \left(I_m(a_1,a_2,\ldots,a_m)\cap T^{-m-n}B)\right)=\mu\left(I_m(a_1,a_2,\ldots,a_m)\right)\mu(B)(1+O({\rho}^n))$$
  for any $m\ge1, n\ge0$, any $m$-th cylinder $I_m(a_1,a_2,\ldots,a_m)$ and any Borel set $B$, where the implied constant in $O({\rho}^n)$ is absolute.
\end{lemma}

Using Lemma \ref{mixing-estimate} and Schmidt's orthogonal method \cite{Schmidt8}, we have the following statement.

\begin{lemma}\label{BClemma}
For all $k\geq1$ let $\CC_k\subseteq\N^d$ and $P_k=\cup_{(a_1, \ldots, a_d)\in\CC_k}I_d$. Then
  $$\mu\left(\left\{x\in [0, 1): T^k(x)\in P_k \ \text{for i.m. } n\in\N\right\}\right)=\left\{ 
\begin{array}{ll}
0& {\rm if}\ \ \sum\limits_{k=1}^\infty\mu (P_k)<\infty; \\ [3ex] 
1 & {\rm if}\ \   \sum\limits_{k=1}^\infty\mu (P_k)=\infty.
\end{array}
\right.
$$
\end{lemma}
Using the fact that Gauss measure $\mu$ is equivalent to Lebesgue measure $\mathcal{L}$, we get
\begin{corollary}\label{BCcorol}
Suppose that $P_n,n\geq1$ are as in Lemma \ref{BClemma}. Then
$$
\mathcal{L}(\{x\in[0,1)\, : \, T^n(x) \in P_n \text{ for i.m. } \, n\in\N \})
$$
is null or full according to whether the series $\sum_{n=1}^\infty \mathcal{L}(P_n)$ converges or not.
\end{corollary}

We also need the following technical lemma, which helps us to choose a suitable set of parameters to construct the Cantor subset.  Consider a non-decreasing sequence of real numbers $1\leq \alpha_1 \leq \ldots \leq \alpha_{d-1}\leq \alpha_d = B$. 
\begin{proposition}\label{pro1}
With the notation as above, we have
\begin{equation}\label{A1An}
\sup_{1\le \alpha_{1}\le \ldots \le \alpha_{d-1}\le B}\min\left\{\alpha_{1}^{-s},\ldots, \left(\alpha_{k-1}^{1-s}\alpha_{k}^{-s}\right)^{\frac{1}{k}},\ldots,
\left(\alpha_{d-1}^{1-s}\cdot B^{-s}\right)^{\frac{1}{d}}\right\}=\frac{1}{B^{g_{d}(s)}}, 
\end{equation}
where $2\le k\le d-1$ and the supremum is attained when all the terms are equal.

\end{proposition}
\begin{proof}

The proof is quite similar with the Proposition 2.13 in \cite{HuWuXu}. 
When $d=2$, this problem is simplified to 
$$\sup_{1\le \alpha_{1}\le B}\min\left\{\alpha_{1}^{-s},\alpha_{1}^{\frac{1}{2}(1-s)}B^{-\frac{s}{2}}\right\}.$$
As a function of $\alpha_{1}$, $\alpha_{1}^{-s}$ is decreasing and $\alpha_{1}^{\frac{1}{2}(1-s)}$ is increasing,  so the supremum is attained when both the terms are equal. Furthermore, it is equal to $B^{-g_{2}(s)}$. Now assume that the claim holds for $d=\ell$. Then we prove it for $d=\ell+1$ as,
\begin{align*}
&\sup_{1\le \alpha_{1}\le \ldots \le \alpha_{\ell}\le B}\min\left\{\alpha_{1}^{-s},\ldots,
\left(\alpha_{\ell-1}^{1-s}\alpha_{\ell}^{-s}\right)^{\frac{1}{\ell}},\left(\alpha_{\ell}^{1-s}\cdot B^{-s}\right)^{\frac{1}{\ell+1}}\right\}\\
=&\sup_{1\le \alpha_{\ell}\le B}\sup_{(\alpha_1,\ldots, \alpha_{\ell-1}): 1\le \alpha_{1}\le \ldots\le \alpha_{\ell}\le B}\min\left\{\min\left\{\alpha_{1}^{-s},\ldots,
\left(\alpha_{\ell-1}^{1-s}\alpha_{\ell}^{-s}\right)^{\frac{1}{\ell}}\right\},\left(\alpha_{\ell}^{1-s}\cdot B^{-s}\right)^{\frac{1}{\ell+1}}\right\}\\
=&\sup_{1\le \alpha_{\ell}\le B}\min\left\{\sup_{(\alpha_1,\ldots, \alpha_{\ell-1}):1\le \alpha_{1}\le \ldots\le \alpha_{\ell}\le B}\min\left\{\alpha_{1}^{-s},\ldots,
\left(\alpha_{\ell-1}^{1-s}\alpha_{\ell}^{-s}\right)^{\frac{1}{\ell}}\right\},\left(\alpha_{\ell}^{1-s}\cdot B^{-s}\right)^{\frac{1}{\ell+1}}\right\}.
\end{align*}
By induction, the following supremum
$$\sup_{(\alpha_1,\ldots, \alpha_{\ell-1}): 1\le \alpha_{1}\le \ldots\le \alpha_{\ell}\le B}\min\left\{\alpha_{1}^{-s},\ldots,
\left(\alpha_{\ell-1}^{1-s}\alpha_{\ell}^{-s}\right)^{\frac{1}{\ell}}\right\},$$
is attained when all terms are equal to $\alpha_{\ell}^{-g_{\ell}(s)}$. Then, as a function of $\alpha_{\ell}$, we obtain that the supremum is attained when two terms $\alpha_{\ell}^{-g_{\ell}(s)}$ and $\left(\alpha_{\ell}^{1-s}\cdot B^{-s}\right)^{\frac{1}{\ell+1}}$ are equal. From this we get
$$
\alpha_\ell = B^\frac{s}{1-s+(\ell+1)g_\ell(s)},
$$
and so the supremum is equal to $\alpha_\ell^{-g_\ell(s)}=B^{-g_{\ell+1}(s)}$.
\end{proof}

From the previous proposition, we have defined a sequence  $(\alpha_{1},\ldots,\alpha_{d-1})$, for which the supremum is attained and we let $\alpha_d= B$.  For this choice of $\alpha$'s, denote $A_{1}=\alpha_{1}$, $A_{i}=\alpha_{i}/\alpha_{i-1}$ for any $i\ge 2$. Hence we can also extract the values of $A_i$, as $\alpha_k = A_1\cdots A_k, 1\leq k \leq d-1$. Note that by definition of $\alpha_{d}$ we have
\begin{equation}\label{AB}
A_{1}\cdots A_{d}=B.
\end{equation}
We have the following properties of the sequence $(A_1,\ldots, A_d)$ defined above.
\begin{proposition}\label{pro2}
 For any $n\ge 1$ and $s>1/2$, the sequence $(A_{1},\ldots,A_{d})$ from the definition above satisfies:
\begin{itemize}
\item For any $1\le k\le d$, one has
\begin{equation}\label{cor1}
A_{1}\cdots A_{k}=A_{1}^{-sk}A_{k}^{1-s}\left(A_{1}\cdots A_{k}\right)^{2s}.
\end{equation}
\item A specific expression of $A_{k}$ for any $1\le k\le d$ is
\begin{equation}\label{cor2}
\log A_{k}=\left(\frac{1}{2-s^{-1}}-\frac{1}{2-s^{-1}}\left(\frac{1-s}{s}\right)^{k}\right)\log A_{1},
\end{equation}
which implies $A_{1}\le A_{2}\le\ldots\le A_{d}$.
\end{itemize}
\end{proposition}
\begin{proof}
As minimum in Proposition \ref{pro1} is attained when all of the terms are equal, using (\ref{AB}) we get that for any $1\le j \leq d-1$ one has
$$\frac{1}{j}\left(\sum_{i=1}^{j-1}(2s-1)\log A_{i}+s\log A_{j}\right)=\frac{1}{j+1}\left(\sum_{i=1}^{j}(2s-1)\log A_{i}+s\log A_{j+1}\right)=s\log A_{1}.$$
This implies
\begin{equation}\label{pfc1}
\log A_{j+1}=\log A_{1}+\frac{1-s}{s}\log A_{j},
\end{equation}
or
\begin{equation}\label{pfc2}
A_{j+1}^{s}=A_{1}^{s}A_{j}^{1-s}.
\end{equation}
Multiplying equations above for $j=1,2,\ldots,k-1$, one obtains (\ref{cor1}) immediately. By iterating \eqref{pfc1} one can easily obtain \eqref{cor2}.
\end{proof}

In particular, for a sequence of parameters $(A_{1},\ldots,A_{d})$ chosen above, we get
\begin{equation}\label{propa1}
A_1  = B^\frac{g_d(s)}{s}.
\end{equation}
We will use this identity throughout the proof.

For a better understanding of the Hausdorff dimension of our set, we provide some useful properties of the function $g_n(s).$
\begin{lemma}\label{continuityGN}
    For any fixed $n\geq1$, $g_n(s)$ is a continuous, strictly increasing function on $[0,1]$, and differentiable on $(0,1).$
\end{lemma}

\begin{proof}
Recall that from the iterative definition \eqref{defGN},
\begin{align*}\label{defGN}
g_1(s) = s, \,\,\,\,\, g_n(s)&=\frac{sg_{n-1}(s)}{1-s+ng_{n-1}(s)},
\end{align*}
we have $g_n(0)=0$ and $g_n(1)=1/n$. Note that $g_n(s)$ is a rational function for any $n\geq1.$ We will prove all the properties by induction which are trivially true for $g_1(s) =s.$ Suppose that the conclusion is true for $n-1$. Then we have
$$
1-s +g_{n-1}\geq 1 - s + g_{n-1}(0) = 1-s >0
$$
    for $s\in[0,1)$ and we also have 
$$
1-s+g_n(s) = g_n(1)=1/n>0
$$
for $s=1.$ This means that the denominator of the rational function $g_n(s)$ is always positive, and hence non-zero on $[0,1]$. 
This implies that $g_n(s)$ is continuous on $[0,1]$ and differentiable in $(0,1)$. We can also write
$$
g'_n(s)=\frac{ng_{n-1}^2(s)+g_{n-1}(s)+s(1-s)g'_{n-1}(s)}{(1-s+ng_{n-1}(s))^2}
$$
for $s\in(0,1)$. Since we know that $g_{n-1}(s)>g_{n-1}(0)=0$ and $g'_{n-1}(s)>0$ for any $s\in(0,1)$, we have that $g'_n(s)>0$ for any $s\in(0,1).$ 
\end{proof}

Next, we provide an explicit expression for the iterative function $g_n(s).$

\begin{proposition}\label{explicitgd}
For any $n\geq1$ we have $g_n(0)=0, g_n(1)=1/n$, and for any $0<s<1$ we have
\begin{equation}\label{explicitGDeq}\frac{1}{g_n(s)}=\left\{
  \begin{array}{ll}
   n(n+1) &{\rm if  }
   \,s=1/2,\\ [3ex] 
   \frac{(1-s)^{n-1}+(2n+1)s^{n+1}-(n+1)s^n}{s^n(2s-1)^2} &{\rm if    }\,s\neq 1/2.\\
  \end{array}
\right.
\end{equation} 
    
\end{proposition}

\begin{proof}
From the previous lemma we know that the value at $s=0$ and $s=1$ is a corollary of the iterative definition
   $$ 
   g_1(s) = s, \,\,\,\,\, g_n(s)=\frac{sg_{n-1}(s)}{1-s+ng_{n-1}(s)}.
   $$
   This definition also implies
$$
\frac{1}{g_n(s)} = \frac{n}{s}+\frac{1-s}{s}\frac{1}{g_{n-1}(s)}.
$$
Now if $s=\frac{1}{2}$, we have
\begin{align*}
    \frac{1}{g_n(1/2)}& = 2n+\frac{1}{g_{n-1}(1/2)}\\
    &=2n+\cdots+2\cdot2+\frac{1}{g_1(1/2)}\\
    &=2n+\cdots+2\cdot2+2\\
    &=n(n+1).
\end{align*}
If $s\neq1/2$, then we have\footnote{We are thankful to an anonymous referee for suggesting the last expression in this formula which is more compact than what we originally had.}
\begin{align*}
\frac{1}{g_n(s)} &= \frac{n}{s}+\frac{1-s}{s}\frac{1}{g_{n-1}(s)}\\
&= \frac{n}{s}+\frac{1-s}{s}\left( \frac{n-1}{s}+\frac{1-s}{s}\frac{1}{g_{n-2}(s)} \right)\\
&=\left(\frac{1-s}{s}\right)^{n-1}\frac{1}{g_{1}(s)}+\sum_{j=0}^{n-2} \left(\frac{1-s}{s}\right)^{j} \frac{n-j}{s}\\
&=\left(\frac{1-s}{s}\right)^{n-1}\frac{1}{s}+\sum_{j=0}^{n-2} \left(\frac{1-s}{s}\right)^{j} \frac{n-j}{s}\\ &=\frac{(1-s)^{n-1}+(2n+1)s^{n+1}-(n+1)s^n}{s^n(2s-1)^2}.
\end{align*}

\end{proof}
From the explicit formula for $1/2 < s < 1$, we easily get
\begin{corollary}\label{decreasingGn} For all $n\geq1$ and any $s\in(1/2,1)$ we have
    $$
g_{n+1}(s)\leq g_n(s) \leq s<1.
$$
\end{corollary}
\begin{proof}
    Using the formula from Proposition \ref{explicitgd} we can show that
    \begin{equation}\label{1231241}
        \frac{1}{g_n(s)} \leq \frac{1}{g_{n+1}(s)}.
    \end{equation}

The second inequality is a consequence of $g_1(s) = s$ and \eqref{1231241}.
\end{proof}

\subsection{Pressure Function}\label{Pressure Functions}
Consider a sub-system $(X_{\mathcal{A}},T)$ of system $([0,1),T)$, define
$$X_{\mathcal{A}}=\Big\{x\in\left[0,1\right):a_{n}\left(x\right)\in\mathcal{A},\textrm{ for all }n\ge 1\Big\},$$
where $\mathcal{A}\subseteq\mathbb{N}$ is finite or infinite. We give the pressure function which is restricted to system $(X_{\mathcal{A}},T)$. Given any real function $\phi:[0,1)\rightarrow\mathbb{R}$, define
\begin{equation}\label{pressure}
\mathsf{P}_{\mathcal{A}}(T,\phi)=\lim_{n\rightarrow\infty}\frac{1}{n}\log \sum_{(a_1,\ldots,a_n)\in\mathcal{A}^{n}}\sup_{x\in X_{\mathcal{A}}} e^{S_n\phi([a_1,\ldots,a_n+x])},
\end{equation}
where $S_n\phi$ denotes the ergodic sum $\phi(x)+\cdots+\phi(T^{n-1}x)$. In case when $\mathcal{A}=\mathbb{N}$, we denote $\mathsf{P}(T,\phi)$ by $\mathsf{P}_{\mathbb{N}}(T,\phi)$. We will use the notation of $n$th variation of $\phi$ to explain the existence of limit (\ref{pressure}), let
$$\Var_n(\phi):=\sup\Big\{|\phi(x)-\phi(y)|:I_n(x)=I_y(y)\Big\}.$$
In \cite{LiWaWuXu}, authors have given a proposition to guarantee the existence of this limit.
\begin{proposition}[{\protect\cite{LiWaWuXu}}]\label{pro11}
Let $\phi:\left[0,1\right)\rightarrow\mathbb{R}$ be a real function with ${\rm{Var}}_{1}\left(\phi\right)<\infty$ and ${\rm{Var}}_{n}\left(\phi\right)\rightarrow 0$ as $n\rightarrow\infty$. Then the limit defining $\mathsf{P}_{\mathcal{A}}(T,\phi)$ exists and the value of $\mathsf{P}_{\mathcal{A}}(T,\phi)$ remains the same even without taking supremum over $x\in X_{\mathcal{A}}$ in (\ref{pressure}).
\end{proposition}
The next proposition states that the pressure function has a continuity property when the system $([0,1),T)$ is approximated by its sub-system $(X_{\mathcal{A}},T)$.
\begin{proposition}[{\protect\cite{HaMaUr}}]\label{pro}
Let $\phi:\left[0,1\right)\rightarrow\mathbb{R}$ be a real function with ${\rm{Var}}_{1}\left(\phi\right)<\infty$ and ${\rm{Var}}_{n}\left(\phi\right)\rightarrow 0$ as $n\rightarrow\infty$. We have
$$\mathsf{P}_{\mathbb{N}}\left(T,\phi\right)=\sup\Big\{\mathsf{P}_{\mathcal{A}}\left(T,\phi\right):\mathcal{A}\textrm{ is a finite subset of }\mathbb{N}\Big\}.$$
\end{proposition}
Let us specify the potential function $\phi := \phi_d$ in our case. Let
\begin{align*}
\phi_d(x)=-s\log |T^{\prime}(x)|-g_d(s)\log B.
\end{align*}
For any $d\geq1$, we denote
$$\lambda_d(B, \mathcal{A}):=\inf\left\{s\ge 0:\mathsf{P}_{\mathcal{A}}\left(-s\log\left|T^{\prime}(x)\right|-g_d(s)\log B\right)\le 0\right\},$$
$$\lambda_d(B,\mathcal{A},n):=\inf\left\{s\ge 0:\sum_{a_{1},\ldots,a_{n}\in\mathcal{A}}\frac{1}{q_{n}^{2s}\left(a_{1},\ldots,a_{n}\right)B^{g_d(s)n}}\le 1\right\}.$$
If $\mathcal{A}$ is a finite subset of $\mathbb{N}$, we can see that series is equal to 1 when $s=\lambda_d(B,\mathcal{A},n)$ or that pressure is equal to 0 when $s=\lambda_d(B, \mathcal{A})$. For simplicity, when $\mathcal{A}=\mathbb{N}$, we denote $\lambda_d(B, \mathcal{A})$ by $\lambda_d(B)$ and $\lambda_d(B,\mathcal{A},n)$ by $\lambda_d(B,n)$. For some integer $M\in\mathbb{N}$, when $\mathcal{A}=\{1,2,\ldots,M\}$, we denote $\lambda_d(B,M,n)$ by $\lambda_d(B,M)$ .

It is clear that $\phi$ satisfies the variation condition. By Proposition \ref{pro}, one has
\begin{proposition}\label{pro2}
For any $M\in\mathbb{N}$, we have
\begin{align*}
\lim_{n\rightarrow\infty}\lambda_d(B,n)=\lambda_d(B),\textrm{  }\lim_{n\rightarrow\infty}\lambda_d(B,M,n)=\lambda_d(B,M),\textrm{  }\lim_{M\rightarrow\infty}\lambda_d(B,M)=\lambda_d(B).
\end{align*}
\end{proposition}

The following proposition states the continuity and asymptotic properties of $\lambda_d(B)$. The proof of this is similar to Lemma 2.6 in \cite{WaWu08}. We provide it for completeness. 
\begin{proposition}\label{prop213}
As a function of $B\in\left(1,\infty\right)$, $\lambda_d(B)$ is continuous and decreasing. Moreover,
$$\lim_{B\rightarrow 1}\lambda_d(B)=1,\textrm{  }\lim_{B\rightarrow\infty}\lambda_d(B)=\frac{1}{2}.$$
\end{proposition}
In particular, we have $\lambda_d(B)\ge\dfrac{1}{2}$. 
\begin{proof}
First, let us show 
$\lim_{B\rightarrow 1}\lambda_d(B)=1$. For a fixed $\varepsilon>0$, take 
$$B_0 = 2^\frac{\varepsilon s}{g_d(s)}>1.$$
As clearly $\lambda_d(B)\leq1$, we can show that for all $1<B<B_0$, we have $\lambda_d(B,n)>1-2\varepsilon$. To get this, we see that for any $n\geq2$
\begin{align*}
    &\sum_{a_{1},\ldots,a_{n}}\frac{1}{\left(q_{n}^{2}B^{\frac{g_d(1-2\varepsilon)}{1-2\varepsilon}n}\right)^{1-2\varepsilon}} 
    \geq \sum_{a_{1},\ldots,a_{n}}\frac{q_n^{4\varepsilon}}{q_{n}^{2}B^{g_d(1-2\varepsilon)n}}\\
    &\geq \sum_{a_{1},\ldots,a_{n}}\frac{2^{n\varepsilon}}{q_{n}(q_{n}+q_{n-1})B^{g_d(1-2\varepsilon)n}} = \left( \frac{2^\varepsilon}{B^{g_d(1-2\varepsilon)}} \right)^{n} >1 ,
\end{align*}
where we used \eqref{length} and \eqref{qE}.

Second, we prove  $\lim_{B\rightarrow\infty}\lambda_d(B)=\frac{1}{2}.$ First, notice that 
$$
\sum_{a_{1},\ldots,a_{n}}\frac{1}{\left(q_{n}^{2}B^{2g_d(1/2)n}\right)^{1/2}} \geq \left(\frac{1}{B^{d(d+1)}} \sum_{k=1}^{\infty} \frac{1}{k+1} \right)^n =\infty,
$$
    where we used \eqref{trivialbound} and Proposition \ref{explicitgd}. This implies that $\lambda_d(B,n)\geq1/2$ for any $n\geq1$ and $B>1$. 
    Now for any $\varepsilon>0$, take 
    $$B_0 = \left(\sum_{k=1}^\infty \frac{1}{k^{1+2\varepsilon}} \right)^{\frac{1}{g_d(1+2\varepsilon)}}<\infty.$$
Then for any $B>B_0$ we have
$$
\sum_{a_{1},\ldots,a_{n}}\frac{1}{\left(q_{n}^{2}B^{\frac{g_d(1/2+\varepsilon)}{1/2+\varepsilon}n}\right)^{1/2+\varepsilon}} \leq 
B^{-ng_d(1/2+\varepsilon)} \left(\sum_{k=1}^\infty \frac{1}{k^{1+2\varepsilon}}\right)^n = \left(\frac{B_0}{B}\right)^{ng_d(1/2+\varepsilon)} <1,
$$
where we used \eqref{trivialbound}.

Finally, we show that $\lambda_d(B)$ is a continuous function. Fix $n\geq2.$
It is clear that $\lambda_d(B,n)$ is a monotonically decreasing function with respect to $B$. For any $\varepsilon>0$, consider $B_0$, such that $1<B_0 < B< B_0+\varepsilon$. We have
\begin{align*}
&\sum_{a_{1},\ldots,a_{n}}\frac{1}{q_{n}^{2(\lambda_d(B,n)+2\varepsilon)} B_0^{g_d(\lambda_d(B,n)+2\varepsilon)n}} \leq \frac{1}{2^{2\varepsilon (n-1)} }\sum_{a_{1},\ldots,a_{n}}\frac{1}{q_{n}^{2\lambda_d(B,n)} B_0^{g_d(\lambda_d(B,n))n}} \\
& \leq \frac{1}{2^{2\varepsilon (n-1)}} \left( \frac{B}{B_0} \right)^{g_d(\lambda_d(B,n))n} \leq \frac{1}{2^{2\varepsilon (n-1)}} \left( \frac{B}{B_0} \right)^{n} <\frac{1}{4^{\varepsilon (n-1)}}  \left( 1+\frac{\varepsilon}{B_0} \right)^{n} < 1.
\end{align*}
Here we used \eqref{qE}, Corollary \ref{decreasingGn} and the inequality $1+x\leq e^x.$  

This means that $\lambda_d(B,n)$ is uniformly continuous with respect to $B$, which is sufficient to have continuity of the limit function $\lambda_d(B).$
\end{proof}

\subsection{Comparison with other dimensional numbers}\label{comparison}
There is a well-known result about the Hausdorff dimension of a set with large consecutive partial quotients. Consider the set
 \begin{align*}
\mathcal{D}_{d}(\Phi)=\left\{x\in[0, 1): \prod_{i=1}^d a_{n+i}(x)\geq \Phi(n) \ \
\mathrm{for \ infinitely \ many \ }n\in\mathbb{N}\right\}.
\end{align*}
The following statement was proven is \cite{HuWuXu}.
\begin{theorem}[Huang-Wu-Xu, 2020]\label{HWXtheorem}
Let $\Phi :\mathbb{N}\rightarrow \mathbb [2, \infty)$. Suppose
\begin{equation}\label{psilim}\log B=\liminf\limits_{n\rightarrow \infty }\frac{\log
\Phi (n)}{n} \quad { and}\quad \log b=\liminf\limits_{n\rightarrow \infty }\frac{\log
\log \Phi (n)}{n}.
\end{equation}
Then, 
\begin{equation*}
\dim_\HH \DD_d(\Phi) =\left\{
\begin{array}{ll}
1 & {\rm if}\ \ B=1; \\ [3ex]
\theta_d(B):= \inf \{ s: \mathsf{P}(T, -f_d(s) \log B -s \log | T^{\prime} | ) \leq 0 \} & {\rm if}\ \ 1<B<\infty; \\ [3ex]
\frac{1}{1+b} & {\rm if} \ \  B=\infty,
\end{array}
\right.
\end{equation*}
where $f_d(s)$ is given by the following iterative formula:
$$
f_1(s)=s, \,\,\,\,\, f_{k+1}(s) = \frac{s f_k (s)}{1-s+f_k(s)}  \text{ for }k \geq 1.
$$
%\end{enumerate}
\end{theorem}

Note that in both $\DD_d(\Phi)$ and $\Lambda_d(\Phi)$, we have a product of $d$ partial quotients. So it is natural to compare their dimensions. We can prove that in the case $1<B<\infty$ the dimension of the set of numbers with large partial quotients in arithmetic progressions is larger than the dimension of the set of continued fractions with large products of consecutive partial quotients.

To do this, first note that we can write down an explicit formula for $f_k$, similarly to the Proposition \ref{explicitgd}.
\begin{proposition}
    For any $n\geq1$ we have $f_n(0)=0, f_n(1)=1$ and for any $0<s<1$ we have
\begin{equation}\label{explicitFNeq}
f_n(s)=\left\{
  \begin{array}{ll}
   \frac{1}{2n} &{\rm if  }
   \,s=1/2,\\ [3ex] 
   \frac{s^n(2s-1)}{s^n-(1-s)^n} &{\rm if    }\,s\neq 1/2.\\
  \end{array}
\right.
\end{equation}     
\end{proposition}
Now using explicit formulae \eqref{explicitGDeq} and \eqref{explicitFNeq}, as an easy exercise one can show that for a fixed $n$, the function
$$
\omega_n(s) = \frac{f_n(s)}{g_n(s)}
$$
is a continuous strictly increasing function of $s$ on $[0,1]$. In particular, this means that for $1/2<s<1$, we have
$$
\omega_n(s) \geq \omega_n(1/2) = \frac{n+1}{2},
$$
so
\begin{equation}\label{fngncomparison}
f_n(s)\geq \frac{n+1}{2}g_n(s)
\end{equation}
on $1/2<s<1.$

 Now, remember that the dimensional numbers are decreasing with respect to $B$, which by inequality \eqref{fngncomparison} implies that
 $$
 \lambda_d(B) > \theta_d(B)
 $$
for $d\geq2.$

Further, note that \eqref{fngncomparison} also implies an estimate valid for all $d\geq 1$ to be  $$
 \lambda_d(B) \geq \theta_d(B^{(d+1)/2}).
 $$
This means that in the case of consecutive partial quotients, we can increase $B$ to $B^{(d+1)/2}$ without surpassing the dimension of $\Lambda_d(B)$.

\section{Proof of Theorem \ref{FurstCor} when $\Phi(n)=B^n$}\label{HDsection}
In this section, we prove the Hausdorff dimension of the set under consideration. We split the case into two parts: the upper bound of the Hausdorff dimension and the lower bound.  Here we deal with the case when $\Phi(n)=B^n$. For simplicity, denote $\Lambda_d(B):=\Lambda_d(B^n)$. We deal with a general function $\Phi(n)$ in Section \ref{generalfunction}.

In the proofs below we will adapt the following convention. If an object or a function depends on a sequence of variables of length $n$ with indices being consecutive integers, say from $1$ to $n$, then we will drop the dependence from the notation. This convention applies to $q_n(a_1,\ldots,a_n),I_n(a_1,\ldots,a_n), J_n(a_1,\ldots,a_n)$ or $g_n(a_1,\ldots,a_n)$, instead of which we will simply write  $q_n,I_n,J_n$ and $g_n$ respectively. Note that the lower index $n$ shows the length of the sequence of variables. Also, if an object depends on variables with consecutive indices starting from a number different from $1$, we specify the starting index, for example 
$$q_{N}(a_3):=q_{N}(a_3,\ldots,a_{N+2}).$$

\subsection{The upper bound} Let $1\leq \beta<B$. Consider a set
$$G_d(\beta, B):=\left\{x\in[0,1):\prod_{i=1}^{d-1}a_{in}(x)\leq \beta^n, \ a_{dn}(x)\geq \frac{B^n}{\prod_{i=1}^{d-1}a_{in}(x)} \ \text{for infinitely many } n\in\N\right\}.$$
Then $$\Lambda_d(B)\subseteq \Lambda_{d-1}(\beta)\cup G_d(\beta, B),$$
and therefore, since Hausdorff measure and dimension is additive, we have
\begin{equation}\label{eq1n}
\dim_\HH \Lambda_d(B)\leq \inf_{1<\beta\leq B} \max\{\dim_\HH\Lambda_{d-1}(\beta), \dim_\HH G_d(\beta, B)\}.
\end{equation}

Let
$$J_{dn-1}:=\bigcup_{a_{dn}(x)\geq \frac{B^n}{\prod_{i=1}^{d-1}a_{in}(x)}}I_{dn}.$$
Then
$$\left|J_{dn-1}\right| \asymp \frac{\prod_{i=1}^{d-1}a_{in}(x)}{B^nq_{dn-1}^2}.$$
We will also make use of the two following facts. First, for a fixed $\varepsilon>0$, there exists $N_0 \in \N$ such that for all $n\ge N_0$, we have
$$
\frac{ (\log \beta^n)^{d-1}}{(d-1)!} \le \frac{ (\log B^n)^{d-1}}{(d-1)!} \le B^{n\varepsilon}.
$$
Next, we have a following useful statement.
\begin{lemma}[\cite{HuWuXu}, Lemma 4.2]\label{lemmahuwuxu}
Let $\beta>1$. For any integer $k\ge1, 0<s<1$, we have
\begin{equation}
\sum\limits_{1\le a_{n}\cdots a_{n+k-1} \le \beta^n} \left( \frac{1}{a_n\cdots a_{n+k-1}} \right)^s \asymp \frac{ (\log \beta^n)^{k-1}}{(k-1)!} \beta^{n(1-s)}.
\end{equation}
\end{lemma}

Using Lemma \ref{lemmahuwuxu} and the inequality above it, we can proceed with the following estimate on Hausdroff measure of the set $G_d(\beta,B)$.

\begin{align*}
\HH^s(G_d(\beta, B))&\ll \sum_{a_{1},\ldots,a_{dn-1}: \prod_{i=1}^{d-1}a_{in}\leq\beta}\left(\frac{\prod_{i=1}^{d-1}a_{in}(x)}{B^nq_{dn-1}^2}\right)^s \\ &\asymp \left(\sum_{a_1,\ldots, a_{n-1}}\frac1{q_{n-1}^{2s}}\right)^d \sum_{a_n\cdots a_{(d-1)n}\leq \beta^n}\left(\frac1{\prod_{i=1}^{d-1} a_{in}}\right)^s\frac{1}{B^{ns}} \\ &\asymp  \frac{\left(\log \beta^n\right)^{d-1}}{(d-1)!}\beta^{n(1-s)}\left(\sum_{a_1,\ldots, a_{n-1}}\frac1{q_{n-1}^{2s}}\right)^d\frac{1}{B^{ns}}\\
&\asymp B^{n\varepsilon}\beta^{n(1-s)}\left(\sum_{a_1,\ldots, a_{n-1}}\frac1{q_{n-1}^{2s}}\right)^d\frac1{B^{ns}}.
\end{align*}
Thus, due to the arbitrariness of $\varepsilon>0$, we have
$$
\dim_\HH G_d(\beta, B) \leq \inf\{ s: P(-s\log|T^\prime(x)|+\frac{1-s}{d}\log\beta-\frac{s}{d}\log B)\leq 0 \}=: s_d(\beta,B).
$$
 By induction we also have
\begin{equation}\label{betan}
\dim_\HH\Lambda_{d-1}(B) = \inf\left\{s: P(-s\log|T^\prime(x)|-g_{d-1}(s)\log\beta)\leq0\right\}=:t_{d-1}(\beta).
\end{equation}
The pressure function $P(T,\phi)$ is increasing with respect to the potential $\psi$, so $s_d(\beta,B)$ is increasing and $t_{d-1}(\beta)$  is decreasing with respect to $\beta$. This means that the infimum in \eqref{eq1n} is attained at the value $\beta$ for which we have the equality

\begin{align*}
-s\log|T^\prime|+\frac{1-s}{d}\log\beta-\frac{s}{d}\log B&=-s\log|T^\prime| - g_{d-1}(s)\log \beta \\ \Longrightarrow \log\beta &=\frac{s}{1-s+dg_{d-1}(s)}\log B.
\end{align*}
Now substituting the expression for $\log \beta$ into the definition \eqref{betan} of $t_{d-1}(\beta)$, we get the following upper bound of the Hausdorff dimension of $\Lambda_d(B)$.
\begin{align*}
\dim_\HH\Lambda_d(B) & \leq   \inf\left\{s: P(-s\log|T^\prime(x)|-\frac{s g_{d-1}(s)}{1-s+dg_{d-1}(s)}  \log B)\leq0\right\} \\
& = \inf\left\{s: P(-s\log|T^\prime(x)|-g_{d}(s)\log B)\leq0\right\}=\lambda_d(B),
\end{align*}
where we used the recursive definition of $g_d(s)$.

\subsection{The lower bound}\label{sectionlower} We choose a rapidly increasing sequence of integers $\{\ell_{k}\}_{k\ge 1}$, fixed integers $M$ and $N$, define $n_{0}=0, n_k-1= \ell_k N$, $m_k=dn_{k}+1+r_k N$ for any $k\ge 1$ such that
\begin{align}\label{requirenk}
\lim_{k\rightarrow\infty}r_k=\infty \text{ and } n_{k+1}-m_k -1\leq 2N.
\end{align}
Consider a following subset of $E$:
\begin{align*}
E_{M}=\Big\{x\in\left[0,1\right):&A_1^{n_{k}}\le a_{n_{k}}\left(x\right)\le 2A_1^{n_{k}},A_2^{n_{k}}\le a_{2n_{k}}\left(x\right)\le 2A_2^{n_{k}},\ldots,A_{d}^{n_{k}}\le a_{dn_{k}}\left(x\right)\le 2A_{d}^{n_{k}},\\
&\phantom{=\;\;}a_{m_{k}+1}(x)=\cdots=a_{n_{k+1}-1}(x)=2 \textrm { for all }k\ge 1 ,1\le a_{n}\left(x\right)\le M\textrm{ for other cases }\Big\},
\end{align*}
where $(A_{1},\ldots, A_d)_{k\ge 1}$ is a sequence for which the minimum in \eqref{pro1} is attained.
\subsubsection{Structure of $E_{M}$}
In this subsection, we will study set $E_M$ in symbolic space which is defined as follows. Let
\begin{align*}
D_n=\Big\{\left(a_{1},\ldots,a_{n}\right)\in\mathbb{N}^{n}:&A_1^{n_{k}}\le a_{n_{k}}\le 2A_1^{n_{k}},A_2^{n_{k}}\le a_{2n_{k}}\le 2A_2^{n_{k}},\ldots,A_{dn_k}^{n_{k}}\le a_{dn_{k}}\le 2A_{dn_k}^{n_{k}},\\
&\phantom{=\;\;}a_{m_k+1}=\cdots=a_{\min\{n_{k+1}-1,n\}}=2 \textrm { for all }k\ge 1 ,1\le a_{n}\le M\textrm{ for other cases }\Big\}
\end{align*}
and $$D=\bigcup_{n=0}^{\infty}D_{n}\quad \left(D_{0}:=\emptyset\right).$$
Write $I_{0}=\left[0,1\right]$. For any $n\ge 1$ and $\left(a_{1},\ldots,a_{n}\right)\in D_{n}$, we call $I_n=I_n\left(a_{1},\ldots,a_{n}\right)$ a $\textit{basic interval of order n}$ and
\begin{equation}\label{J}
J_n=J_n\left(a_{1},\ldots,a_{n}\right)=\bigcup_{a_{n+1}}\textrm{cl} I_{n+1}
\end{equation}
a $\textit{fundamental interval of order n}$, where the union in $\left(\ref{J}\right)$ is taken over all $a_{n+1}$ such that $\left(a_{1},\ldots,a_{n},a_{n+1}\right)\in~D_{n+1}$.
It is clear that
\begin{equation}
E_{M}=\bigcap_{n\ge 1}\bigcup_{\left(a_{1},\ldots,a_{n}\right)\in D_{n}}J_{n}.
\end{equation}
\subsubsection{Lengths of fundamental intervals}
In this subsection, we calculate the lengths of fundamental intervals while splitting them into several distinct cases, which we will use later.
\begin{enumerate}[(1)]
\item When $n=jn_{k}-1$ for any $k\ge 1$ and $1\le j\le d$, we have
\begin{align*}
J_{jn_{k}-1}=\bigcup_{A_{j}^{n_{k}}\le a_{jn_{k}}\le 2A_{j}^{n_{k}}}I_{jn_{k}}.
\end{align*}
Therefore we get
\begin{align*}
\left|J_{jn_{k}-1}\right|&=\frac{A_{j}^{n_{k}}+1}{\left(\left(2A_{j}^{n_{k}}+1\right)q_{jn_{k}-1}+q_{jn_{k}-2}\right)\left(A_{j}^{n_{k}}q_{jn_{k}-1}+q_{jn_{k}-2}\right)}
\end{align*}
and hence
\begin{equation}\label{J1}
\frac{1}{2^{3}A_{j}^{n_{k}}q_{jn_{k}-1}^{2}}\le\left|J_{jn_{k}-1}\right|\le\frac{1}{A_{j}^{n_{k}}q_{jn_{k}-1}^{2}}.
\end{equation}
\item When $m_k\le n\le n_{k+1}-2$ for any $k\ge 1$, we have
\begin{align*}
J_{n}= I_{n+1}\left(a_{1},\ldots,a_{m_k},2,\ldots,2\right).
\end{align*}
where the number of partial quotients that are equal to 2 is $n-m_k+1$.
Therefore,
\begin{align*}
\left|J_{n}\right|&=\frac{1}{(2q_n+q_{n-1})(3q_n+q_{n-1})}
\end{align*}
and
\begin{equation}\label{J2}
\frac{1}{2^{4(n-m_k)+5}q_{m_k-1}^{2}}\le\frac{1}{2^{5}q_{n}^{2}}\le\left|J_{n}\right|\le\frac{1}{2^{2}q_{n}^{2}}\le \frac{1}{2^{2(n-m_k)+2}q_{m_k-1}^{2}}.
\end{equation}

\item For other cases, one has
\begin{align*}
J_{n}=\bigcup_{1\le a_{n+1}\le M}I_{n+1}.
\end{align*}
Therefore,
\begin{align*}
\left|J_{n}\right|&=\frac{M}{\left(\left(M+1\right)q_{n}+q_{n-1}\right)\left(q_{n}+q_{n-1}\right)}
\end{align*}
and
\begin{equation}\label{J3}
\frac{1}{6q_{n}^{2}}\le\left|J_{n}\right|\le\frac{1}{q_{n}^{2}}.
\end{equation}
\end{enumerate}

\subsubsection{Gaps in $E_{M}$}
For a given a fundamental interval $J_{n}\left(a_{1},\ldots,a_{n}\right)$, where $\left(a_{1},\ldots,a_{n}\right)\in D_{n}$, we estimate the gap between $J_{n}$ and its adjacent fundamental intervals of the same order $n$. Without loss of generality, we assume $n$ is even. Let $g_{n}^{l}=g_{n}^{l}\left(a_{1},\ldots,a_{n}\right)$ and $g_{n}^{r}=g_{n}^{r}\left(a_{1},\ldots,a_{n}\right)$ are defined respectively as the left and the right minimum distance between $J_{n}$ and its adjacent fundamental intervals at each side. Write
\begin{align*}
g_{n}=\min\left\{g_{n}^{l},g_{n}^{r}\right\}.
\end{align*}
Since $n$ is even, we note that $g_{n}=g_{n}^{r}$.  In addition, we only consider the case of $J_{n}$ and its right adjacent fundamental intervals of order $n$ are all in $I_{n-1}$, otherwise consider their parents fundamental interval.
\begin{enumerate}[(1)]
\item When $n=jn_{k}-1$ for any $k\ge 1$ and $1\le j\le d$, one will see that $g_{n}$ is larger than the distance between the right endpoint of $I_{jn_{k}}\left(a_{1},\ldots,a_{jn_{k}-1},A_{j}^{n_{k}}\right)$ and the right endpoint of $I_{jn_{k}-1}$. Thus

\begin{equation}\label{G1}
\begin{split}
g_{n}&\ge\frac{p_{jn_{k}-1}+p_{jn_{k}-2}}{q_{jn_{k}-1}+q_{jn_{k}-2}}-\frac{A_{j}^{n_{k}}p_{jn_{k}-1}+p_{jn_{k}-2}}{A_{j}^{n_{k}}q_{jn_{k}-1}+q_{jn_{k}-2}}\ge\frac{1}{4A_{j}^{n_{k}}q_{jn_{k}-1}^{2}}\\
&\ge \frac{1}{4}\left|J_{n}\right|\quad\textrm{by $\left(\ref{J1}\right)$}.
\end{split}
\end{equation}

\item When $m_k\le n\le n_{k+1}-1$ for any $k\ge 1$, one has $g_{n}$ is larger than the distance between the right endpoint of $I_{n+1}\left(a_{1},\ldots,a_{n},2\right)$ and the right endpoint of $I_{n}$. Thus
\begin{equation}\label{G2}
\begin{split}
g_{n}&\ge \frac{p_n+p_{n-1}}{q_n+q_{n-1}}-\frac{2p_n+p_{n-1}}{2q_n+q_{n-1}}\ge \frac{1}{6q_{n}^{2}}\\
&\ge \frac{1}{6}\left|J_{n}\right|\quad\textrm{by $\left(\ref{J2}\right)$}.
\end{split}
\end{equation}

\item For other cases, by assumption, $g_{n}^{r}$ is just the distance between the right endpoint of $I_{n}$ and the left endpoint of $I_{n+1}\left(a_{1},\ldots,a_{n}+1,M\right)$.
\begin{equation}\label{G3}
\begin{split}
g_{n}&=\frac{\left(M+1\right)p_{n}+\left(M+2\right)p_{n-1}}{\left(M+1\right)q_{n}+\left(M+2\right)q_{n-1}}-\frac{p_{n}+p_{n-1}}{q_{n}+q_{n-1}}\ge\frac{1}{10Mq_{n}^{2}}\\
&\ge \frac{1}{10M}\left|J_{n}\right|\quad\textrm{by $\left(\ref{J3}\right)$}.
\end{split}
\end{equation}

\end{enumerate}

\subsubsection{The mass distribution of $E_{M}$}
We define a measure $\mu$ supported on $E_{M}$. The mass on fundamental intervals of order $n_{k},2n_{k},\ldots,dn_{k}$ depends on the total number of possible values of $a_{n_{k}},a_{2n_{k}},\ldots,a_{dn_{k}}$ for any $k\ge 1$. For intervals of order $n$ with  $m_k+1\le n\le n_{k+1}-1$ for any $k\ge 1$, the mass is same as their parents fundamental interval. As for other intervals, we define uniform distribution and use consistency of measure. More specifically, given any integers $N,M$ large enough, denote 
\begin{align*}
u:=\sum_{(a_{1},\ldots,a_{N})\in\{1,2,\ldots,M\}^{N}}\frac{1}{q_{N}^{2s}B^{g_d(s)N}}.
\end{align*}

For any $s<\lambda_d(B,M,n)$, one has $u\geq1$. Let $n_{k}-1=\ell_{k}N$ for any $k\ge 1$, where $n_{0}=0$. We define mass by induction.
\begin{enumerate}[(1)]
\item For any $(a_{1},\ldots,a_{iN})\in D_{iN}$ and $1\le i<\ell_1$, let
\begin{align*}
\mu\left(J_{iN}\right)=\prod_{t=0}^{i-1}\frac{1}{u}\cdot\frac{1}{q_{N}^{2s}\left(a_{tN+1}\right)B^{g_d(s)N}}.
\end{align*}
\item For any $\left(a_{1},\ldots,a_{jn_{1}}\right)\in D_{jn_{1}}$ and $1\le j\le d$, let
\begin{align*}
&\quad\mu\left(J_{jn_{1}}\right)=\prod_{i=1}^{j}\frac{1}{A_{i}^{n_{1}}}\prod_{t=0}^{\ell_{1}-1}\frac{1}{u}\cdot\frac{1}{q_{N}^{2s}\left(a_{\left(i-1\right)n_{1}+tN+1}\right)B^{g_d(s)N}}.
\end{align*}
\item For any $\left(a_{1},\ldots,a_{jn_{1}-1}\right)\in D_{jn_{1}-1}$ and $1\le j\le d$, let
\begin{align*}
\mu\left(J_{jn_{1}-1}\right)=A_{j}^{n_{1}}\mu\left(J_{jn_{1}}\right).
\end{align*}
\item For any $1\le j< d$, when $n=jn_{1}+iN$ where $1\le i<\ell_{1}$, for any $\left(a_{1},\ldots,a_{n}\right)\in D_{n}$, let
\begin{align*}
\mu\left(J_{jn_{1}+iN}\right)&=\mu\left(J_{jn_{1}}\right)\cdot\prod_{t=0}^{i-1}\frac{1}{u}\cdot\frac{1}{q_{N}^{2s}\left(a_{jn_{1}+tN+1}\right)B^{g_d(s)N}}.
\end{align*}
When $n=dn_1+iN$ for some $1\le i\le r_1$, for any $\left(a_{1},\ldots,a_{n}\right)\in D_{n}$, let
\begin{align*}
\mu(J_{dn_1+iN})&=\mu(J_{dn_1})\cdot\prod_{t=0}^{i-1}\frac{1}{u}\cdot\frac{1}{q_{N}^{2s}\left(a_{dn_1+tN+1}\right)B^{g_d(s)N}}.
\end{align*}
\item For any $m_{1}+1\le n\le n_{2}-1$ and $(a_{1},\ldots,a_n)\in D_{n}$, let
\begin{align*}
&\quad\mu(J_{n})=\mu(J_{m_{1}}).
\end{align*}

\item For other cases, when $n\ne jn_{1}+iN$ for any $0\le j< d$, $0\le i< \ell_{1}$ and $n\notin [m_1+1,n_2-1]$,  there is a unique integer $i$ such that $jn_{1}+(i-1)N\le n<jn_{1}+iN$, for any $\left(a_{1},\ldots,a_{n}\right)\in D_{n}$, let
  \begin{align*}
  \mu(J_{n})\le\mu(J_{jn_{1}+(i-1)N}).
  \end{align*}
  When $n\ne dn_1+iN$ for any $0\le i\le r_{1}$, there is a unique integer $i$ such that $dn_1+(i-1)N\le n<dn_1+iN$, for any $\left(a_{1},\ldots,a_{n}\right)\in D_{n}$, let
  \begin{align*}
  \mu(J_{n})\le\mu(J_{dn_1+(i-1)N}).
  \end{align*}
%%%%%%%%%%%%%%%%%%%%%%%%%%%%%%%%%%%

As for the measure of other fundamental intervals, we adopt induction. Assume the mass on fundamental intervals of the order $n\leq n_{k}-1$ has been defined, then
\item For some $k\ge 2$, for any $(a_{1},\ldots,a_{jn_{k}})\in D_{jn_{k}}$ and any $1\le j\le d$, let
\begin{align*}
\quad\mu\left(J_{jn_{k}}\right) = \mu(J_{n_{k}-1})   \cdot \frac{1}{A_{1}^{n_{k}}}\prod_{i=2}^{j}\frac{1}{A_{i}^{n_{k}}}\prod_{t=0}^{\ell_{k}-1}\frac{1}{u}\cdot\frac{1}{q_{N}^{2s}\left(a_{\left(i-1\right)n_{k}+tN+1}\right)B^{g_d(s)N}}.
\end{align*}
\item For some $k\ge 2$, for any $(a_{1},\ldots,a_{jn_{k}-1})\in D_{jn_{k}-1}$ and any $1< j\le d$, let
\begin{align*}
&\quad\mu\left(J_{jn_{k}-1}\right)=A_{j}^{n_{k}}\mu\left(J_{jn_{k}}\right).
\end{align*}
\item For some $k\ge 2$ and some $1\le j< d$, when $n=jn_{k}+iN$ for some $1\le i< \ell_{k}$, for any $\left(a_{1},\ldots,a_{n}\right)\in D_{n}$, let
\begin{align*}
\mu\left(J_{jn_{k}+iN}\right)&=\mu(J_{jn_{k}})\cdot\prod_{t=0}^{i-1}\frac{1}{u}\cdot\frac{1}{q_{N}^{2s}\left(a_{jn_{k}+tN+1}\right)B^{g_d(s)N}}.
\end{align*}
When $n=dn_{k}+iN$ for some $1\le i\le r_k$, for any $\left(a_{1},\ldots,a_{n}\right)\in D_{n}$, let
\begin{align*}
\mu(J_{dn_{k}+iN})&=\mu(J_{dn_{k}})\cdot\prod_{t=0}^{i-1}\frac{1}{u}\cdot\frac{1}{q_{N}^{2s}\left(a_{dn_{k}+tN+1}\right)B^{g_d(s)N}}.
\end{align*}
\item For some $k\ge 2$, for any $m_{k}+1\le n\le n_{k+1}-1$ and $(a_{1},\ldots,a_n)\in D_{n}$, let
\begin{align*}
&\quad\mu(J_{n})=\mu(J_{m_{k}}).
\end{align*}
\item For other cases, when $n\ne jn_{k}+iN$ for any $k\ge 1$, $0\le j< d$, $0\le i< \ell_{k}$ and $n\notin [m_k+1,n_{k+1}-1]$,  there is a unique integer $i$ such that $jn_{k}+(i-1)N\le n<jn_{k}+iN$, for any $\left(a_{1},\ldots,a_{n}\right)\in D_{n}$, let
  \begin{align*}
  \mu(J_{n})\le\mu(J_{jn_{k}+(i-1)N}).
  \end{align*}
  When $n\ne dn_{k}+iN$ for any $0\le i\le r_{k}$, there is a unique integer $i$ such that $dn_{k}+(i-1)N\le n<dn_{k}+iN$, for any $\left(a_{1},\ldots,a_{n}\right)\in D_{n}$, let
  \begin{align*}
  \mu(J_{n})\le\mu(J_{dn_{k}+(i-1)N}).
  \end{align*}
\end{enumerate}

\subsubsection{Estimation of $\mu(J_n)$}

We compare the mass with the length of fundamental interval and we will frequently use the properties $\left(\ref{c1}\right)-\left(\ref{qE}\right)$ of convergents in the process.
\begin{enumerate}[(1)]
\item When $n=iN$, for some $1\le i<\ell_{1}$, we have
\begin{equation}\label{E1}
\begin{split}
\mu\left(J_{iN}\right)&\le\prod_{t=0}^{i-1}\frac{1}{q_{N}^{2s}\left(a_{tN+1}\right)B^{g_d(s)N}}\\
&\le 2^{2\left(i-1\right)}\cdot\frac{1}{q_{iN}^{2s}}\quad\left(\textrm{by $\left(\ref{qE}\right)$}\right)\\
&\le\left(\frac{1}{q_{iN}^{2}}\right)^{s-\frac{2}{N}}\\
&\ll \left|J_{iN}\right|^{s-\frac{2}{N}}\quad\left(\textrm{by $\left(\ref{J3}\right)$}\right).
\end{split}
\end{equation}
\item When $n=\ell_{1}N=n_{1}-1$, As in the $\left(\ref{E1}\right)$, we have
\begin{equation*}
\begin{split}
\mu\left(J_{n_{1}-1}\right)&\le\prod_{t=0}^{\ell_{1}-1}\frac{1}{q_{N}^{2s}\left(a_{tN+1}\right)B^{g_d(s)N}}\\
&\le 2^{2\left(\ell_{1}-1\right)}\left(\frac{1}{q_{n_{1}-1}^{2}}\right)^{s}\frac{1}{B^{\ell_{1}g_d(s)N}}\quad\left(\textrm{by $\left(\ref{c2}\right)$}\right)\\
&\le\left(\frac{1}{q_{n_{1}-1}^{2}}\right)^{s-\frac{2}{N}}\frac{1}{B^{\ell_{1}g_d(s)N}}\quad\left(\textrm{similar with $\left(\ref{E1}\right)$}\right).\\
&\le\left(\frac{1}{q_{n_{1}-1}^{2} B^{n_1 \frac{g_d(s)}{s}}}\right)^{s-\frac{2}{N}}.
\end{split}
\end{equation*}
Thus, using \eqref{J1} we conclude that
\begin{equation}\label{E2}
\begin{split}
\mu\left(J_{n_{1}-1}\right)\ll \left|J_{n_{1}-1}\right|^{s-\frac{2}{N}}.
\end{split}
\end{equation}

\item When $n=jn_{1}$ for any fixed $1\le j\le d$, we get the following estimation.
\begin{equation*}
\begin{split}
\mu\left(J_{jn_{1}}\right)&\le\prod_{i=1}^{j}\frac{1}{A_{i}^{n_{1}}}\prod_{t=0}^{\ell_{1}-1}\frac{1}{q_{N}^{2s}\left(a_{\left(i-1\right)n_{1}+tN+1}\right)B^{g_d(s)N}}\\
&\le\frac{1}{\left(A_{1}\cdots A_{j}\right)^{n_{1}}B^{j(n_{1}-1)g_d(s)}}\prod_{i=1}^{j}\left(\frac{1}{q_{n_{1}-1}^{2}\left(a_{\left(i-1\right)n_{1}+1}\right)}\right)^{s-\frac{2}{N}}\quad\left(\textrm{similar to case (2)}\right)\\
&=\frac{1}{A_{1}^{-sn_{1}j}A_{j}^{\left(1-s\right)n_{1}}\left(A_{1}\cdots A_{j}\right)^{2sn_{1}}B^{j(n_{1}-1)g_d(s)}}\prod_{i=1}^{j}\left(\frac{1}{q_{n_{1}-1}^{2}\left(a_{\left(i-1\right)n_{1}+1}\right)}\right)^{s-\frac{2}{N}},
\end{split}
\end{equation*}
where the last equality is hold by (\ref{cor1}). Notice that
\begin{align*}
q_{jn_{1}}\le 2^{(2j-1)+j}(A_{1}\cdots A_{j})^{n_{1}}\prod_{i=1}^{j}q_{n_{1}-1}\left(a_{\left(i-1\right)n_{1}+1}\right).
\end{align*}
Thus,
\begin{align}\label{jn1}
|J_{jn_{1}}|\ge \frac{1}{6}\frac{1}{2^{6j-2}(A_{1}\cdots A_{j})^{2n_{1}}}\prod_{i=1}^{j}\frac{1}{q_{n_{1}-1}^{2}\left(a_{\left(i-1\right)n_{1}+1}\right)}.
\end{align}
From \eqref{propa1} we see that $A_1^{-sj} A_j^{1-s} B^{j g_d(s)}=A_j^{1-s} >1$. Therefore, 
\begin{equation*}
\begin{split}
\mu\left(J_{jn_{1}}\right)& \leq \frac{B^{j g_d(s)}}{\left(A_{1}\cdots A_{j}\right)^{2sn_{1}}}\prod_{i=1}^{j}\left(\frac{1}{q_{n_{1}-1}^{2}\left(a_{\left(i-1\right)n_{1}+1}\right)}\right)^{s-\frac{2}{N}} \\
&\leq B^{j g_d(s)} [6 2^{6j-2}]^{s-\frac{8}{N}} \left(\frac{1}{[6 2^{6j-2}] \left(A_{1}\cdots A_{j}\right)^{2n_{1}}}\prod_{i=1}^{j}\frac{1}{q_{n_{1}-1}^{2}\left(a_{\left(i-1\right)n_{1}+1}\right)}\right)^{s-\frac{2}{N}}.
\end{split}
\end{equation*}
So we come to
\begin{equation}\label{E3}
\mu\left(J_{jn_{1}}\right)\ll\left|J_{jn_{1}}\right|^{s-\frac{2}{N}}.
\end{equation}

\item When $n=jn_{1}-1$ for $1<j\le d$, we note that   
$$\frac{\left|J_{jn_{1}}\right|}{\left|J_{jn_{1}-1}\right|}\le \frac{2^{3}A_{j}^{n_1}q_{jn_{1}-1}^{2}}{q_{jn_{1}}^{2}}\le\frac{2^3}{A_{j}^{n_{1}}}.$$
We know that
\begin{equation}\label{E15}
\begin{split}
\mu\left(J_{jn_{1}-1}\right)&=A_{j}^{n_{1}}\mu\left(J_{jn_{1}}\right)\\
&\le\frac{A_{j}^{n_1}}{A_{1}^{-sn_{1}j}A_{j}^{\left(1-s\right)n_{1}}\left(A_{1}\cdots A_{j}\right)^{2sn_{1}}B^{j(n_{1}-1)g_d(s)}}\prod_{i=1}^{j}\left(\frac{1}{q_{n_{1}-1}^{2}\left(a_{\left(i-1\right)n_{1}+1}\right)}\right)^{s-\frac{2}{N}}.
\end{split}
\end{equation}
Multiply the right hand side by $\frac{A_j^{sn_1}}{A_j^{sn_1}}$ and note that $A_j^{-1+s} A_1^{-sj} A_j^{1-s} B^{j g_d(s)}=1$. Therefore, proceeding as in the previous case, we come to
\begin{equation}\label{E4}
\mu\left(J_{jn_{1}-1}\right)\ll\left|J_{jn_{1}-1}\right|^{s-\frac{2}{N}}.
\end{equation}
\item When $n=jn_1+iN$ for some $1\le j< d$ and $1\le i< \ell_1$, we have
\begin{equation}\label{E5}
\begin{split}
\mu\left(J_{jn_{1}+iN}\right)&\le\prod_{t=0}^{i-1}\frac{1}{q_{N}^{2s}\left(a_{jn_{1}+tN+1}\right)B^{g_d(s)N}} \cdot\mu\left(J_{jn_{1}}\right)\\
&\le\left(\frac{1}{q_{iN}^{2}(a_{jn_1+1})}\right)^{s-\frac{2}{N}}\cdot\left(\frac{1}{q_{jn_1}^{2}}\right)^{s-\frac{2}{N}}\\
&\ll \left(\frac{1}{q_{jn_1+iN}^{2}}\right)^{s-\frac{2}{N}}\\
& \ll |J_{jn_{1}+iN}|^{s-\frac{2}{N}}.
\end{split}
\end{equation}

\item When $n=dn_1+iN$ for some $1\le i\le r_1$, we have
\begin{equation}\label{E6}
\begin{split}
\mu\left(J_{dn_1+iN}\right)&\le\prod_{t=0}^{i-1}\frac{1}{q_{N}^{2s}(a_{dn_1+tN+1})B^{g_d(s)N}}\cdot\mu(J_{dn_1})\\
&\le\left(\frac{1}{q_{iN}^{2}(a_{dn_1+1})}\right)^{s-\frac{8}{N}}\cdot\left(\frac{1}{q_{dn_1}^{2}}\right)^{s-\frac{2}{N}}\\
&\ll \left(\frac{1}{q_{dn_1+iN}^{2}}\right)^{s-\frac{2}{N}}\\ &\ll |J_{dn_1+iN}|^{s-\frac{2}{N}}.
\end{split}
\end{equation}

\item When $m_1+1\le n\le n_{2}-1$, we have
\begin{equation*}
\mu(J_{n})=\mu(J_{m_1}).
\end{equation*}
Note that
\begin{align*}
\frac{|J_{m_1}|}{|J_{n}|}\le 2^{4N+5} \textrm{ and } q_{n}^{2}\ge 2^{4N^{2}+5N},
\end{align*}
when $n_1$ is large enough. Thus, one concludes that
\begin{equation}\label{E7}
\mu(J_{n})\ll\left|J_{n}\right|^{s-\frac{60}{N}}.
\end{equation}

\item When $n\ne jn_{1}+iN$ for any $0\le j< d$, $0\le i< \ell_{1}$ and $n\notin [m_1+1,n_2-1]$,  there is a unique integer $i$ such that $jn_{1}+(i-1)N\le n<jn_{1}+iN$, we observe that
\begin{equation*}
\begin{split}
\frac{\left|J_{n}\right|}{\left|J_{jn_1+(i-1)N}\right|}&\ge \frac{1}{6}\frac{q_{jn_1+(i-1)N}^{2}}{q_{n}^{2}}\\
&\ge\frac{1}{6}\left(\frac{1}{2^{n-jn_1-(i-1)N}a_{jn_1+(i-1)N+1}\cdots a_{n}}\right)^{2}\\ &\ge\frac{1}{6}\left(\frac{1}{4M^{2}}\right)^{N}.
\end{split}
\end{equation*}
Thus, one has
\begin{equation}\label{E8}
\begin{split}
\mu(J_{n})&\le\mu(J_{jn_{1}+(i-1)N})\\ &\le\left|J_{jn_{1}+(i-1)N}\right|^{s-\frac{2}{N}}\\
&\le 6\cdot\left(4M^{2}\right)^{N}\left|J_{n}\right|^{s-\frac{2}{N}}\\ &\ll\left|J_{n}\right|^{s-\frac{60}{N}}.
\end{split}
\end{equation}
With the same method, when $n\ne dn_1+iN$ for any $0\le i\le r_{1}$, there is a unique integer $i$ such that $dn_1+(i-1)N\le n<dn_1+iN$, one has
\begin{equation}\label{E8r}
\mu(J_{n})\ll\left|J_{n}\right|^{s-\frac{60}{N}}.
\end{equation}

%%%%%%%%%%%%%%%%%%%%%%%%

\item When $n=n_{k}-1$, we have
\begin{equation*}
\begin{split}
\mu(J_{n_{k}-1})&=\mu(J_{m_{k-1}})\\
&\le\frac{1}{A_{1}^{n_{k-1}}}\prod_{t=0}^{r_{k-1}-1}\left(\frac{1}{q_{N}^{2s}(a_{dn_{k-1}+tN+1})B^{g_d(s)N}}\right)\cdot\prod_{j=1}^{d-1}\frac{1}{A_{j}^{n_{k-1}}}\\
&\quad\cdot\prod_{j=1}^{d-1}\prod_{t=0}^{\ell_{k-1}-1}\left(\frac{1}{q_{N}^{2s}(a_{jn_{k-1}+tN+1})B^{g_d(s)N}}\right)\cdot\mu(J_{n_{k-1}-1})\\
&\le\prod_{b=1}^{k-1}\left(\prod_{j=1}^{d-1}\frac{1}{A_{j}^{n_{b}}}\right)\cdot\prod_{t=0}^{r_{b}-1}\left(\frac{1}{q_{N}^{2s}(a_{dn_{b}+tN+1})B^{g_d(s)N}}\right)\\
&\quad\cdot\prod_{j=1}^{d-1}\prod_{t=0}^{\ell_{b}-1}\left(\frac{1}{q_{N}^{2s}(a_{jn_{b}+tN+1})B^{g_d(s)N}}\right)\cdot\mu(J_{n_{1}-1})\\
&\le\prod_{b=1}^{k-1}\left(\prod_{j=1}^{d-1}\frac{1}{A_{j}^{n_{b}}}\right)\cdot2^{2(r_b-1)}\left(\frac{1}{q_{m_d-dn_{b}-1}^{2s}(a_{dn_{b}+1})B^{g_d(s)r_bN}}\right)\\
&\quad\cdot\prod_{j=1}^{d-1}2^{2(\ell_b-1)}\left(\frac{1}{q_{n_b-1}^{2s}(a_{jn_{b}+1})B^{g_d(s)\ell_dN}}\right)\cdot\left(\frac{1}{q_{n_1-1}^{2}}\right)^{s-\frac{2}{N}}\\
&\le2^{2\sum_{b=1}^{k-1}(r_b-1)+2\sum_{b=1}^{k-1}(\ell_b-1)(n_b-1)}\frac{1}{B^{g_d(s)\sum_{b=1}^{k-1}m_{b}-2n_{b}+1}}\\
& \cdot \prod_{b=1}^{k-1}\left(\frac{1}{q_{m_d-dn_{b}-1}^{2s}(a_{dn_{b}+1})}\prod_{j=1}^{d-1}\frac{1}{q_{n_{b}-1}^{2s}(a_{(j+1)n_{b}+1})}\right)\\
&\quad\cdot\prod_{b=1}^{k-1}\frac{1}{{A_{1}^{-sdn_{b}}A_d}^{(1-s)n_d}(A_{1}\cdots A_d)^{2sn_b}}\cdot\left(\frac{1}{q_{n_1-1}^{2}}\right)^{s-\frac{2}{N}},
\end{split}
\end{equation*}
where the last inequality holds by (\ref{cor1}). Notice that
\begin{equation*}
\begin{split}
q_{n_{k}-1}&\le 2^{2d(k-1)+d(k-1)}\cdot\prod_{b=1}^{k-1}A_{1}^{n_{b}}\cdots A_{d}^{n_{b}}\cdot \prod_{b=1}^{k-1}q_{n_{b+1}-dn_{b}-1}(a_{dn_{b}+1})\cdot \prod_{b=1}^{k-1}\prod_{j=0}^{d-1}q_{n_{b}-1}(a_{jn_{b}+1})\\
&\le 2^{2d(k-1)+d(k-1)+2N} \cdot\prod_{b=1}^{k-1}A_{1}^{n_{b}}\cdots A_{d}^{n_{b}}\cdot \prod_{b=1}^{k-1}q_{m_{b}-dn_{b}-1}(a_{dn_{b}+1}) \cdot \prod_{b=1}^{k-1}\prod_{j=0}^{d-1}q_{n_{b}-1}(a_{jn_{b}+1})
\end{split}
\end{equation*}
and
\begin{align*}
|J_{n_{k}-1}|&\ge \frac{1}{8A_{1}^{n_{k}}q_{n_{k}-1}^{2}}\\ &\ge \frac{1}{8\cdot 2^{6d(k-1)}A_{1}^{n_{k}}}\prod_{b=1}^{k-1}\frac{1}{(A_{1}\cdots A_{d})^{2n_{b}}}\cdot \\
&\prod_{b=1}^{k-1}\frac{1}{q_{m_{b}-dn_{b}-1}^{2}(a_{dn_{b}+1})}\cdot \prod_{b=1}^{k-1}\prod_{j=0}^{d-1}\frac{1}{q_{n_{b}-1}^{2}(a_{jn_{b}+1})}.
\end{align*}
Combining everything together and using the fact that $A_1=B^{\frac{g_d(s)}{s}}$,
we have
\begin{equation}\label{E9}
\mu(J_{n_{k}-1})\le |J_{n_{k}-1}|^{s-\frac{120}{N}}.
\end{equation}

\item When $n=jn_{k}$ for any $k\ge 2$ and $1\le j\le d$, we have
\begin{align*}
\mu\left(J_{jn_{k}}\right) &\le\frac{1}{A_{1}^{n_{k}}}\prod_{i=2}^{j}\frac{1}{A_{i}^{n_{k}}}\prod_{t=0}^{\ell_{k}-1}\cdot\frac{1}{q_{N}^{2s}\left(a_{\left(i-1\right)n_{k}+tN+1}\right)B^{g_d(s)N}}\cdot\mu(J_{n_{k}-1})\\
&\le\frac{1}{A_{1}^{-sjn_{k}}A_{j}^{(1-s)n_{k}}(A_{1}\cdots A_{j})^{2sn_{k}}}\cdot\frac{1}{B^{(n_{k}-1)g_d(s)(j-1)}}\\
&
\cdot \prod_{i=1}^{j-1}\left(\frac{1}{q_{n_{k}-1}^{2}(a_{in_{k}+1})}\right)^{s-\frac{2}{N}}\quad\left(\frac{1}{A_{1}^{n_{k}}q_{n_{k}-1}^{2}}\right)^{s-\frac{120}{N}}.
\end{align*}

Notice that
$$q_{jn_{k}}\le 2^{2(j-1)+1+j}A_{1}^{n_{k}}\cdots A_{j}^{n_{k}}\prod_{i=1}^{j-1}q_{n_{k}-1}(a_{in_{k}+1})\cdot q_{n_{k}-1}$$
and
\begin{align*}
|J_{jn_k}|\ge \frac{1}{6q_{jn_{k}}^{2}}\ge \frac{1}{6\cdot 2^{6j-2}}\frac{1}{(A_{1}\cdots A_{j})^{2n_k}}\prod_{i=1}^{j-1}\frac{1}{q_{n_{k}-1}^{2}(a_{in_{k}+1})}\cdot \frac{1}{q_{n_{k}-1}^{2}}.
\end{align*}
Using \eqref{propa1} and $A_j^{1-s}A_1^s > 1$ we come to
\begin{equation}\label{E10}
\mu(J_{jn_{k}})\le |J_{jn_{k}}|^{s-\frac{120}{N}}.
\end{equation}

\item When $n=jn_{k}-1$ for any $k\ge 2$ and $1<j\le d$, we have
\begin{align*}
\mu(J_{jn_{k}-1})&= A_{j}^{n_{k}}\mu(J_{jn_{k}}).\\
\end{align*}
Notice that
\begin{align*}
\frac{|J_{jn_{k}}|}{|J_{jn_{k}-1}|}\le\frac{2^3}{A_{j}^{n_{k}}}.
\end{align*}
and
\begin{align*}
\mu(J_{jn_{k}-1})&= A_{j}^{n_k}\cdot\mu(J_{jn_{k}})\\
&\le\frac{A_{j}^{n_k}}{A_{1}^{-sjn_{k}}A_{j}^{(1-s)n_{k}}(A_{1}\cdots A_{j})^{2sn_{k}}}\cdot\frac{1}{B^{(n_{k}-1)g_d(s)(j-1)}}\prod_{i=1}^{j-1}\left(\frac{1}{q_{n_{k}-1}^{2}(a_{in_{k}+1})}\right)^{s-\frac{2}{N}}\\
&\quad\cdot\left(\frac{1}{A_{1}^{n_{k}}q_{n_{k}-1}^{2}}\right)^{s-\frac{120}{N}}.
\end{align*}
Combined with \eqref{propa1} and $A_1^{s}>1$, one has
\begin{equation}\label{E11}
\mu(J_{jn_{k}-1})\le |J_{jn_{k}-1}|^{s-\frac{120}{N}}.
\end{equation}

\item When $n=jn_{k}+iN$ for some $1\le j< d$ and $1\le i< \ell_{k}$, we have
\begin{equation}\label{E12}
\begin{split}
\mu(J_{jn_{k}+iN})&\le\prod_{t=0}^{i-1}\frac{1}{q_{N}^{2s}\left(a_{jn_{k}+tN+1}\right)B^{g_d(s)N}} \cdot\mu(J_{jn_{k}})\\
&\le \frac{1}{B^{g_d(s)iN}}\cdot\left(\frac{1}{q_{iN}^{2}\left(a_{jn_{k}+1}\right)}\right)^{s-\frac{2}{N}}\cdot\left(\frac{1}{q_{jn_{k}}^{2}}\right)^{s-\frac{120}{N}}\\
&\ll |J_{jn_{k}+iN}|^{s-\frac{130}{N}}.
\end{split}
\end{equation}
\item When $n=dn_{k}+iN$ for any $1\le i\le r_{k}$, we have
\begin{equation}\label{E13}
\begin{split}
\mu(J_{dn_{k}+iN})&\le\prod_{t=0}^{i-1}\frac{1}{q_{N}^{2s}\left(a_{dn_{k}+tN+1}\right)B^{g_d(s)N}} \cdot\mu(J_{dn_{k}})\\
&\le \frac{1}{B^{g_d(s)iN}}\cdot\left(\frac{1}{q_{iN}^{2}\left(a_{dn_{k}+1}\right)}\right)^{s-\frac{2}{N}} \cdot\left(\frac{1}{q_{dn_{k}}^{2}}\right)^{s-\frac{120}{N}}\\
&\ll |J_{dn_{k}+iN}|^{s-\frac{130}{N}}.
\end{split}
\end{equation}
\item When $m_k+1\le n\le n_{k+1}-1$, we have
\begin{equation*}
\mu(J_{n})=\mu(J_{m_{k}}).
\end{equation*}
Note that
\begin{align*}
\frac{|J_{m_k}|}{|J_{n}|}\le 2^{4N+5} \textrm{ and } q_{n}^{2}\ge 2^{4N^{2}+5N},
\end{align*}
when $n_k$ is large enough. Thus, one concludes that
\begin{equation}\label{E14}
\mu(J_{n})\ll\left|J_{n}\right|^{s-\frac{130}{N}}.
\end{equation}
\item When $n\ne jn_{k}+iN$ for any $k\ge 1$, $0\le j< d$, $0\le i< \ell_{k}$ and $n\notin [m_k+1,n_{k+1}-1]$,  there is a unique integer $i$ such that $jn_{k}+(i-1)N\le n<jn_{k}+iN$, we observe that
\begin{equation*}
\begin{split}
\frac{\left|J_{n}\right|}{\left|J_{jn_k+(i-1)N}\right|}&\ge \frac{1}{6}\frac{q_{jn_k+(i-1)N}^{2}}{q_{n}^{2}}\\ &\ge\frac{1}{6}\left(\frac{1}{2^{n-(i-1)N-jn_k}a_{jn_k+(i-1)N+1}\cdots a_{n}}\right)^{2}\\
&\ge\frac{1}{6}\left(\frac{1}{4M^{2}}\right)^{N}.
\end{split}
\end{equation*}
Then, we have
\begin{equation}\label{E15}
\begin{split}
\mu\left(J_{n}\right)&\le\mu\left(J_{jn_k+(i-1)N}\right)\le\left|J_{jn_k+(i-1)N}\right|^{s-\frac{120}{N}}\\
&\le 6\cdot\left(4M^{2}\right)^{N}\left|J_{n}\right|^{s-\frac{120}{N}}\ll\left|J_{n}\right|^{s-\frac{130}{N}}\quad\left(\textrm{by (\ref{E13}) and $\left(\ref{J3}\right)$}\right).
\end{split}
\end{equation}
With the same method, when $n\ne dn_{k}+iN$ for any $0\le i\le r_{k}$, there is a unique integer $i$ such that $dn_{k}+(i-1)N\le n<dn_{k}+iN$, one has
\begin{equation}\label{E15r}
\mu(J_{n})\ll |J_{n}|^{s-\frac{130}{N}}.
\end{equation}
\end{enumerate}
In a word, for any $n\ge 1$, we conclude
\begin{equation}\label{E16}
\mu(J_{n})\ll |J_{n}|^{s-\frac{130}{N}}.
\end{equation}

\subsubsection{Estimation of $\mu\left(B\left(x,r\right)\right)$}
We consider the measure of a general ball $B(x,r)$ with $x\in E_M$ and $r$ small. Let $n$ be the integer such that
\begin{align*}
g_{n+1}\le r<g_{n}.
\end{align*}
By the definition of $g_{n}$, it is clear that $B\left(x,r\right)$ can only intersect one fundamental interval of order $n$, which is $J_{n}$.

\begin{enumerate}[(1)]
\item When $n=jn_{k}-1$ for any $k\ge 1$ and $1\le j\le d$, considering whether the radius of the ball sometimes is larger than the length of basic interval of order $jn_{k}$ or not, we divide this proof into two parts.
\begin{enumerate}[(i)]
\item When $r\le\left|I_{jn_{k}}\right|$, the ball $B\left(x,r\right)$ can intersect at most four basic intervals of order $jn_{k}$, which is $I_{jn_{k}}\left(a_{1},\ldots,a_{jn_{k}}-1\right)$, $I_{jn_{k}}\left(a_{1},\ldots,a_{jn_{k}}\right)$, $I_{jn_{k}}\left(a_{1},\ldots,a_{jn_{k}}+1\right)$ and $I_{jn_{k}}\left(a_{1},\ldots,a_{jn_{k}}+2\right)$. Then by $\left(\ref{E16}\right)$ and $\left(\ref{G1}\right)$
\begin{equation}
\begin{split}
\mu\left(B\left(x,r\right)\right)&\le4\mu\left(J_{jn_{k}}\right)\ll 4\left|J_{jn_{k}}\right|^{s-\frac{130}{N}}\le 40\left|g_{jn_{k}}\right|^{s-\frac{130}{N}} \le 40r^{s-\frac{230}{N}}.
\end{split}
\end{equation}

\item  When $r>\left|I_{jn_{k}}\right|$, by calculating the number of intersections between the ball and fundamental intervals of order $jn_{k}$, we obtain the mass of the ball. Note that
\begin{align*}
\left|I_{jn_{k}}\right|=\frac{1}{q_{jn_{k}}\left(q_{jn_{k}}+q_{jn_{k}-1}\right)}\ge\frac{1}{2q_{jn_{k}}^{2}}\ge\frac{1}{32A_{j}^{n_{k}}q_{jn_{k}-1}^{2}},
\end{align*}
so we estimate that the number of fundamental interval of order $n_{k}$ contained in $J_{jn_{k}-1}$ that the ball $B\left(x,r\right)$ intersects is at most
\begin{align*}
2r\cdot 32A_{j}^{n_{k}}q_{jn_{k}-1}^{2}+2\le 64rA_{j}^{n_{k}}q_{jn_{k}-1}^{2}.
\end{align*}
Thus, by $\left(\ref{E16}\right)$
\begin{equation}
\begin{split}
\mu\left(B\left(x,r\right)\right)&\le\min\left\{\mu\left(J_{jn_{k}-1}\right),64rA_{j}^{n_{k}}q_{jn_{k}-1}^{2}\mu\left(J_{jn_{k}}\right)\right\}\\
&\le\mu\left(J_{jn_{k}-1}\right)\min\left\{1,64rA_{j}^{n_{k}}q_{jn_{k}-1}^{2}\frac{1}{A_{j}^{n_{k}}}\right\}\\
&\ll\left|J_{jn_{k}-1}\right|^{s-\frac{130}{N}}\min\left\{1,64rq_{jn_{k}-1}^{2}\right\}\\
&\ll\left(\frac{1}{q_{jn_{k}-1}^{2}}\right)^{s-\frac{130}{N}}\left(64rq_{jn_{k}-1}^{2}\right)^{s-\frac{130}{N}}\\
&\ll r^{s-\frac{130}{N}}.
\end{split}
\end{equation}
\end{enumerate}

\item When $m_k\le n\le n_{k+1}-2$ for any $k\ge 1$, consider $r>\left|I_{n+1}\right|$, by calculating the number of intersections between the ball and fundamental intervals of order $n+1$, we obtain the mass of the ball. Note that
\begin{align*}
\left|I_{n+1}\right|=\frac{1}{q_{n+1}\left(q_{n+1}+q_{n}\right)}\ge\frac{1}{12q_{n}^{2}},
\end{align*}
so we estimate the number of fundamental interval of order $n+1$ contained in $J_{n}$ that the ball $B\left(x,r\right)$ intersects is at most
\begin{align*}
2r\cdot 12q_{n}^{2}+2\le 24rq_{n}^{2}.
\end{align*}
Thus, by $\left(\ref{E16}\right)$
\begin{equation}
\begin{split}
\mu\left(B\left(x,r\right)\right)&\le\min\left\{\mu\left(J_{n}\right),24rq_{n}^{2}\mu\left(J_{n+1}\right)\right\}\\
&\le\mu\left(J_{n}\right)\min\left\{1,24rq_{n}^{2}\right\}\\
&\ll\left|J_{n}\right|^{s-\frac{130}{N}}\min\left\{1,24rq_{n}^{2}\right\}\\
&\ll\left(\frac{1}{q_{n}^{2}}\right)^{s-\frac{130}{N}}\left(24rq_{n}^{2}\right)^{s-\frac{130}{N}}\\
&\ll r^{s-\frac{130}{N}}.
\end{split}
\end{equation}
%\end{enumerate}

\item When $n\ne jn_{k}-1$ for any $k\ge 1$, we know that $1\le a_{n}\le M$  for any $n\ge 1$ and $\left|J_{n}\right|\asymp\dfrac{1}{q_{n}^{2}}$. Then by $\left(\ref{E16}\right)$ and $\left(\ref{G3}\right)$, we have
\begin{align*}
\mu\left(B\left(x,r\right)\right)&\le\mu\left(J_{n}\right)\ll\left|J_{n}\right|^{s-\frac{130}{N}}\\
&\le\left(\frac{1}{q_{n}^{2}}\right)^{s-\frac{130}{N}}\le 4M^{2}\left(\frac{1}{q_{n+1}^{2}}\right)^{s-\frac{130}{N}}\\
&\le 24M^{2}\left|J_{n+1}\right|^{s-\frac{130}{N}}\\
&\le 240M^{3}\left|g_{n+1}\right|^{s-\frac{130}{N}}\\
&\le 240M^{3}r^{s-\frac{130}{N}}.
\end{align*}
\end{enumerate}
By Proposition $\ref{MD}$, we conclude that
\begin{align*}
\dim_{H}E_{M}\ge s-\frac{130}{N}.
\end{align*}
Letting $N\rightarrow\infty$ and then $M\rightarrow\infty$ implies that
\begin{align*}
\dim_{H}E\ge \lambda_d(B).
\end{align*}

\section{Completing the proof for general function}\label{generalfunction}

In this section we focus on the dimension of $\Lambda_d\left(\Phi\right)$ for a general function $\Phi(n)$. First, we introduce a result we will use later. More details about this result can be found in L\"uczak [12].
\begin{lemma}\label{41}
For any $b,c>1$, sets
$$\Big\{x\in[0,1):a_{n}\left(x\right)\ge c^{b^{n}}\textrm{ for infinitely many }n\Big\}$$
 and 
 $$\Big\{x\in[0,1):a_{n}\left(x\right)\ge c^{b^{n}}\textrm{ for sufficiently large }n\Big\}$$
  share the same Hausdorff dimension $1/\left(b+1\right)$.
\end{lemma}
Now we are ready to prove Theorem \ref{FurstCor} that we split into several parts.
\begin{enumerate}[(1)]
\item For the case $B=1$, we observe that
$$
\Lambda_d\left(\Phi\right)\supseteq\Big\{x\in\left[0,1\right):a_{n}\left(x\right)\ge \Phi\left(n\right)\textrm{ for infinitely many }n\in\mathbb{N}\Big\}.
$$
By the result of Wang and Wu [3], we obtain that $\dim_{H}\Lambda_d\left(\Phi\right)=1$.

\item For the case $1<B<\infty$, for any $0<\varepsilon<B-1$ we have $\Phi\left(n\right)>\left(B-\varepsilon\right)^{n}$ when $n$ is large enough. Thus, we have
$$
\Lambda_d\left(\Phi\right)\subseteq\Big\{x\in\left[0,1\right):\prod_{i=1}^{d}a_{in}\left(x\right)\ge\left(B-\varepsilon\right)^{n}\textrm{ for infinitely many }n\in\mathbb{N}\Big\}.
$$
On the other hand, one can choose a sparse integer sequence $\{ n_j \}_{j\geq1}$, such that for all $j$ we have $$\Phi\left(n_j\right)<\left(B+\varepsilon\right)^{n_j}.$$
Then similar to Section \ref{sectionlower} we will prove the lower bound by constructing and analysing set, analogous to the set $E_M$ from the lower bound section. The only difference is that here we are given the sequence $n_j$ as oppose to the lower bound proof, where we have chosen the sequence on our own. So in general we cannot guarantee that $n_k-1$ is a multiple of $N$. However, we can express $n_k = \ell_k N + r_k$, where $0 \leq r_k < N$. Afterwards, we restrict partial quotients in blocks of the length $N$ to be from the set $\{1,\ldots, M\}$ and we set partial quotients in positions corresponding to the remainder $r_k$ to be equal to $2$.\\
The rigorous proof of the lower bound for this set can be carried out from Section \ref{sectionlower} with no more changes. Now one can conclude 
$$
\lambda_d(B+\varepsilon)\le\dim_{H}\Lambda_d\left(\Phi\right)\le \lambda_d(B-\varepsilon).
$$
Letting $\varepsilon\rightarrow 0$, we obtain $\dim_{H}\Lambda_d\left(\Phi\right)=\lambda_d(B)$.

\item For the case $B=\infty$, we consider the following three cases.
\begin{enumerate}[(i)]
\item For $1<b<\infty$, for any $\varepsilon>0$ we have $\Phi\left(n\right)>e^{\left(b-\varepsilon\right)^{n}}$ for sufficiently large $n$. We observe that 
$$
\prod_{i=1}^{d}a_{in}\left(x\right)\ge e^{(b-\varepsilon)^{n}}
$$ implies that there exists $1\le i\le d$ such that $a_{in}\left(x\right)\ge e^{(b-\varepsilon)^{n}/d}$ for any $x\in\Lambda_d\left(\Phi\right)$. We note that 
$$
e^{(b-\varepsilon)^{n}/d}\geq e^{(b-2\varepsilon)^{n}}.
$$
Thus, we obtain an upper bound of $\Lambda_d\left(\Phi\right)$.
$$\Lambda_d\left(\Phi\right)\subseteq\Big\{x\in\left[0,1\right):a_{n}\left(x\right)\ge e^{\left(b-2\varepsilon\right)^{n}}\textrm{ for infinitely many }n\in\mathbb{N}\Big\}.$$
Therefore, by Lemma \ref{41} we have
$$\dim_{H}\Lambda_d\left(\Phi\right)\le\frac{1}{1+b-2\varepsilon}.$$
By the arbitrary of $\varepsilon>0$, we have
$$\dim_{H}\Lambda_d\left(\Phi\right)\le\frac{1}{1+b}.$$
As for the lower bound, we observe that for any $\varepsilon>0$, we have $\Phi\left(n\right)<e^{\left(b+\varepsilon\right)^{n}}$ for infinitely many $n$. Let
$$\mathcal{G}=\Big\{n:\Phi\left(n\right)<e^{\left(b+\varepsilon\right)^{n}}\Big\}.$$
Then we have
\begin{align*}
\Lambda_d\left(\Phi\right)&\supseteq\Big\{x\in\left[0,1\right):\prod_{i=1}^{d}a_{in}\left(x\right)\ge e^{\left(b+\varepsilon\right)^{n}}\textrm{ for infinitely many }n\in\mathcal{G}\Big\}\\
&\supseteq\Big\{x\in\left[0,1\right):a_{in}\left(x\right)\ge e^{\left(b+\varepsilon\right)^{in}},1\le i\le d,\textrm{ for infinitely many }n\in\mathbb{N}\Big\}\\
&\supseteq\Big\{x\in\left[0,1\right):a_{n}\left(x\right)\ge e^{\left(b+\varepsilon\right)^{n}},\textrm{ for all }n\in\mathbb{N}\Big\}.
\end{align*}
Thus, by the arbitrary of $\varepsilon>0$ we have
$$\dim_{H}\Lambda_d\left(\Phi\right)=\frac{1}{1+b}$$
by Lemma \ref{41}.
\item For $b=1$, the proof of the lower bound is same as the case $1<b<\infty$, but in the end we take $b=1$. This implies that $\dim_{H}\Lambda_d\left(\Phi\right)\ge 1/2$. As for the upper bound, we note that $B=\infty$, so we have
$$\frac{\log\Phi\left(n\right)}{n}>\log B_{1}$$
for a sufficiently large $B_{1}>0$. This implies $\Phi\left(n\right)>B_{1}^{n}$ holds for sufficiently large $n$. Then, we have
$$\Lambda_d\left(\Phi\right)\subseteq\Big\{x\in\left[0,1\right):\prod_{i=1}^{d}a_{in}\left(x\right)\ge B_{1}^{n}\textrm{ for infinitely many }n\in\mathbb{N}\Big\}.$$
Thus
$$\dim_{H}\Lambda_d\left(\Phi\right)\le \lambda_d(B_1).$$
Letting $B_{1}\rightarrow\infty$ and recalling Proposition \ref{prop213}, we obtain the desired result.
\item For $b=\infty$, for a sufficiently large $B_{2}>0$, we have
$$\frac{\log\log\Phi\left(n\right)}{n}\ge\log B_{2}.$$
Thus $\Phi\left(n\right)>e^{B_{2}^{n}}$ holds for sufficiently large $n$. Then similarly to the case (i) we have
$$\Lambda_d\left(\Phi\right)\subseteq\Big\{x\in\left[0,1\right):a_{ni}\left(x\right)\ge e^{(B_{2}^{1/i}/d^{1/i})^{ni}}\textrm{ for infinitely many }n\in\mathbb{N}\Big\}$$
for some $1\leq i \leq d$.
This implies that
$$\dim_{H}\Lambda_d\left(\Phi\right)\le\frac{1}{(B_{2}^{1/i}/d^{1/i})+1}\rightarrow 0$$
as $B_{2}\rightarrow\infty$.
\end{enumerate}

\end{enumerate}

\section{Proof of Theorem \ref{Lebthm}}
For each $n\geq1$ consider the set
$$
\Lambda_{d,n-1}(\Phi) = \{ x\in [0,1) \, : \, a_1(x) a_{n+1}(x)\cdots a_{(d-1)n+1} \geq \Phi(n) \}.
$$
It is clear that $x\in \Lambda_{d}(\Phi)$ if and only if $T^{n-1}(x)\in \Lambda_{d,n-1}(\Phi)$ for infinitely many $n\geq1$.
Thus by Lemmas \ref{mixing-estimate} and \ref{BClemma}, it suffices to see whether the Lebesgue measure the series $\sum_n \mathcal{L}(\Lambda_{d,n}(\Phi))$ is convergent or not.
We get
\begin{align*}
\mathcal{L}(\Lambda_{d,n-1}(\Phi))&\asymp\sum_{a_{1},\ldots,a_{(d-1)n+1}:a_{1}\cdots a_{(d-1)n+1}\ge \Phi(n)}\frac{1}{q_{(d-1)n+1}^{2}}\\
&\asymp\sum_{a_{1},\ldots,a_{(d-1)n+1}:a_{1}\cdots a_{(d-1)n+1}\ge \Phi(n)}\frac{1}{a_{1}^{2}q_{n-1}^{2}(a_{2})a_{n+1}^{2}\cdots q_{n-1}^{2}(a_{(d-2)n+2})a_{(d-1)n+1}^{2}}\\
&=\left(\sum_{a_{1},\ldots,a_{n-1}}\frac{1}{q_{n-1}^{2}}\right)^{d}\sum_{a_{1}\cdots a_{(d-1)n+1}\ge \Phi(n)}\frac{1}{a_{1}^{2}a_{n+1}^{2}\cdots a_{(d-1)n+1}^{2}}\\
&\asymp\sum_{a_{1}\cdots a_{(d-1)n+1}\ge \Phi(n)}\frac{1}{a_{1}^{2}a_{n+1}^{2}\cdots a_{(d-1)n+1}^{2}}.
\end{align*}
In the last line, we used the fact that the term $\left(\sum_{a_{1},\ldots,a_{n-1}}\frac{1}{q_{n-1}^{2}}\right)^{d}$ is at least $1/2^d$ and at most $2^d$. Note that as $d$ is finite, omitting this factor will not change the convergence.\\
Finally, from the proof of Theorem 1.5 in \cite{HuWuXu} we know that
$$
\sum_{(a_1,a_2\ldots,a_d)\in \N^d : a_{1}\cdots a_{d}\geq \Phi(n) } \prod_{k=1}^d \frac{1}{a_k^2} \asymp \frac{\log^{d-1} \Phi(n)}{\Phi(n)},
$$
where the implied constant depends only on $d$.\\
This means that 
$$ 
\sum_{n=1}^\infty  \mathcal{L}( \Lambda_{d,n}(\Phi)) = \infty \qquad\Longleftrightarrow \qquad \sum_{n=1}^\infty \frac{\log^{d-1} \Phi(n)}{\Phi(n)} = \infty.
$$
Application of Corollary \ref{BCcorol} finishes the proof.

\end{document}